\documentclass[a4, 11pt]{amsart}
\usepackage{amsmath,
  amssymb,epic,graphicx,mathrsfs,enumerate,enumitem,subfig, a4}
\usepackage[all]{xy}
\usepackage{color}
\usepackage{comment}

\usepackage{amsthm}
\usepackage{amssymb}
\usepackage{latexsym}
\usepackage{longtable}
\usepackage{epsfig}
\usepackage{amsmath}
\usepackage{hhline}

\DeclareMathOperator{\soc}{soc}
\DeclareMathOperator{\aut}{Aut}

\DeclareMathOperator{\frat}{\Phi}

\DeclareMathOperator{\oo}{O}

\DeclareMathOperator{\SU}{SU}
\DeclareMathOperator{\sym}{S}

\DeclareMathOperator{\psl}{L}
\DeclareMathOperator{\psp}{S}
 
\DeclareMathOperator{\gl}{GL}
\DeclareMathOperator{\psu}{U}

\DeclareMathOperator{\AGL}{AGL}

\DeclareMathOperator{\alt}{A}

\DeclareMathOperator{\fit}{Fit}

\newcommand{\FF}{\mathbb F}

\renewcommand{\emptyset}{\varnothing}

\newcommand{\st}{such that }
\newcommand{\ifa}{if and only if }

\renewcommand{\leq}{\leqslant}
\renewcommand{\geq}{\geqslant}

\renewcommand{\le}{\leqslant}
\renewcommand{\ge}{\geqslant}

\newtheorem{thm}{Theorem}[section]
\newtheorem{mainthm}{Theorem}
\newtheorem{cor}[thm]{Corollary}
 \newtheorem{lemma}[thm]{Lemma}
\newtheorem*{lemma*}{Lemma}
\newtheorem{prop}[thm]{Proposition}

\theoremstyle{definition}
\newtheorem{eg}[thm]{Example}
\newtheorem{assump}[thm]{Assumption}
\newtheorem{notation}[thm]{Notation}
\numberwithin{equation}{section}

\newcommand{\ignore}[1]{}\makeglossary

\begin{document}
	\bibliographystyle{amsplain}
\title[Finite groups satisfying the independence property]{Finite groups satisfying\\ the independence property}

\author[S. D. Freedman]{Saul D. Freedman}
\address{Saul D. Freedman\\ School of Mathematics and Statistics\\ University of St Andrews\\ St Andrews, KY16 9SS, UK\\email: sdf8@st-andrews.ac.uk}
\author[A. Lucchini]{Andrea Lucchini}
\address{Andrea Lucchini\\ Universit\`a di Padova\\  Dipartimento di Matematica \lq\lq Tullio Levi-Civita\rq\rq\\ Via Trieste 63, 35121 Padova, Italy\\email: lucchini@math.unipd.it}
\author[D. Nemmi]{Daniele Nemmi}
\address{Daniele Nemmi\\ Universit\`a di Padova\\  Dipartimento di Matematica \lq\lq Tullio Levi-Civita\rq\rq\\ Via Trieste 63, 35121 Padova, Italy\\email: dnemmi@math.unipd.it}
\author[C. M. Roney-Dougal]{Colva M. Roney-Dougal}
\address{Colva M. Roney-Dougal\\ School of Mathematics and Statistics\\ University of St Andrews\\ St Andrews, KY16 9SS, UK\\email: colva.roney-dougal@st-andrews.ac.uk}


\begin{abstract} We say that a finite group $G$ satisfies the \emph{independence
  property} if, for every pair of distinct elements $x$  and $y$ of $G$,
  either $\{x,y\}$ is contained in a minimal generating set for $G$ or
  one of $x$ and $y$ is a power of the other. We give a complete
  classification of the finite groups with this property, and in particular prove that 
  every such group is supersoluble. A key ingredient of our proof is a
    theorem showing that all but three finite almost simple groups $H$ contain an element $s$ such that the maximal subgroups of $H$ containing $s$, but not  containing the socle of $H$, are pairwise non-conjugate. 	\end{abstract}

\maketitle

\section{Introduction}

Let $G$ be a finite group. A generating set $X$ for $G$ is said to be \emph{minimal} if no proper subset of $X$ generates $G.$ Let $d(G)$ and $m(G)$ denote, respectively, the smallest and largest cardinality of a minimal generating set for $G$. A nice result in universal
algebra, due to Tarski and known as the Tarski Irredundant Basis Theorem (see, for example, \cite[Theorem 4.4]{busa}), implies
that, for every positive integer $k$ with $d(G)\leq k\leq m(G),$ the group $G$ has a minimal generating set of cardinality $k$. However, minimal generating sets for finite groups are not well understood. In particular, while several results in the literature (e.g., \cite{guralnick,quesgen,tracey}) yield good estimates for $d(G)$, very little is known about $m(G)$. An exhaustive investigation \cite{cc,whi} was carried out for the finite
symmetric groups, proving that $m(\sym_n) = n-1$ for each $n$, and giving a complete
description of the minimal generating sets of $\sym_n$ having cardinality $n-1$. The problem of determining $m(G)$ in general remains open, even for finite simple groups, though partial results for certain families of these groups are given in \cite{sawhi}.

One natural related question is ``which subsets of $G$ lie in a minimal generating set?'' For singletons, the answer is easy: an element belongs to some minimal generating set for $G$ if and only if it is not contained in the Frattini subgroup of $G$. Therefore, the first meaningful question is ``which pairs of distinct elements belong to a minimal generating set?'' Similarly, we can ask ``which pairs of distinct elements belong to a generating set of size $d(G)$?'' A partial answer to the corresponding question about singletons is given in \cite[\S6]{al}, using \cite[Theorem 1]{bgh}.

We will call two distinct
elements $x$ and $y$ of $G$ {\emph{independent}} in $G$ if there
exists a minimal generating set $X$ for $G$ with 
$\{x, y\} \subseteq X.$ Similarly, we will call $x$ and $y$
{\emph{rank-independent}} in $G$ if there exists such an $X$ with $|X|=d(G)$.  
An obvious obstruction to $x$ and $y$  being independent is that one of the two is a power of the other. We say that $G$
satisfies the {\emph{independence property}} if this is the unique
obstruction, i.e., if two distinct elements $x$ and $y$ are independent whenever neither of $x$ and $y$ is a power of the
other. Similarly, an obvious obstruction
to $x$ and $y$  being rank-independent is that 
they generate a cyclic subgroup, i.e., that each of $x$ and $y$ is a power of some $z \in G$. We say that a non-cyclic finite group $G$ satisfies the {\emph{rank-independence property}} if $\{x,y\}$ extends to a generating set for $G$ of size $d(G)$ whenever $\langle x, y\rangle$ is not cyclic.

Note that we can also formulate the independence and rank-independence properties in the context of certain graphs associated with $G$. 
Each graph defined here has vertex set $G$. 
In the {\emph{independence graph}} of $G$, which was introduced and investigated in
\cite{ind}, 
two distinct vertices are adjacent if and only if they are
independent, while in the {\emph{rank graph}} of $G$, two distinct vertices 
are adjacent if and only if they are rank-independent. The \emph{power graph} of $G$, where distinct vertices are adjacent if and only if one is a power of the other, was introduced by Kelarev and
Quinn \cite{kq} and investigated by several authors (see for example \cite{c0,c1,c3,c4,c6}). Finally, the edges of the \emph{enhanced power graph} of $G$ are the pairs $\{x,y\}$ of distinct vertices such that $\langle x, y \rangle$ is cyclic. This graph was introduced to interpolate between the power graph and the well-known
commuting graph, but has since 
been studied in its own right
(see \cite{c5,e1,z1,z2}). The independence property of a group is equivalent to its independence graph being
the complement of its power graph. Similarly, $G$ satisfies the rank-independence property if and only if its rank graph is the complement of its enhanced power graph.

In this paper, we will give a complete classification of the finite
groups $G$ satisfying the independence property and those satisfying the rank-independence property. We will see in particular that
in each case $G$ is supersoluble.

The classification of finite groups with the rank-independence property is not particularly
difficult (however, our proof relies on the classification of finite simple groups). The description depends on whether $d(G)=2$ or $d(G)\geq 3$. 
 
\begin{mainthm}\label{thm:rank_perfect2}
 Let $G$ be a finite group with $d(G)=2$. Then $G$ satisfies the rank-independence property if
 and only if one of the following occurs:
 \begin{enumerate}[label={(\roman*)},font=\upshape]
		\item $G\cong C_p\times C_p$, with $p$ a prime;
		\item $G\cong Q_8$; or
		\item $G \cong C_p \rtimes C_{q^m}$, where $p$ and $q$ are distinct primes, $m$ is an arbitrary positive integer, and the action of $C_{q^m}$ on $C_p$ has kernel $C_{q^{m-1}}$.
	\end{enumerate}
\end{mainthm}

\begin{mainthm}\label{thm:rank_perfect3}
  Let $G$ be a finite group with $d(G)\geq 3.$ Then $G$ satisfies the rank-independence property if and only if $G=P\rtimes C,$ where $P$ is an elementary abelian $p$-subgroup of $G$ and $C$ is a  cyclic group of coprime order
    acting on $P$ as scalar multiplication.
\end{mainthm}
In the above result, we permit $C = 1$, and more generally, $G = P \times C$.

In a very recent paper, Harper \cite{sh}  introduced the notion of
$k$-flexible groups, for each positive integer 
$k$. A finite group $G$ is
\emph{$k$-flexible} if, for all $g_1,\dots,g_k\in G$ such that $d(\langle g_1,\dots,g_k\rangle)=k,$ there exist
$g_{k+1},\dots,g_{d(G)} \in G$ such that
$\langle g_1,\dots,g_{d(G)} \rangle = G$.
In particular, the notions of 2-flexible
groups and groups with the rank-independence property coincide.
Theorems~\ref{thm:rank_perfect2} and~\ref{thm:rank_perfect3}
above correspond to a small correction of Lemma 2.7 and a slightly more precise statement
of part of Theorem 2.14 in \cite{sh},
respectively, and were proved independently.

The classification of the finite groups with the independence property
is much more difficult. To prove that a group satisfying the
independence property is supersoluble, we require several new tools that
rely on classifications of the finite simple groups and their maximal subgroups, including the
following key result.  For an element $s$ of a group $G$, we
write $\mathcal{M}_G(s)$ to denote the set of maximal subgroups of $G$
containing $s$. 

\begin{mainthm}
\label{thm:noncomint}
Let $S$ be a non-abelian finite simple group. Then there exist non-commuting elements $s,x \in S$ such that, for each almost simple group $G$ with socle $S$, the intersection $\bigcap_{M \in \mathcal{M}_G(s)} M$ contains $x$.
\end{mainthm}

We shall prove Theorem~\ref{thm:noncomint} in
Section~\ref{sec:nonconjsub} as a consequence of two stronger
theorems, which may be of interest in their own right:
Theorem~\ref{thm:maxsubnoconj} below, which deals with all but three
choices for $S$, and Theorem~\ref{thm:o8plusexceps}, which addresses
the remaining groups.

For an almost simple group $G$ with socle $S$, we let
$\mathcal{M}'_G(s)$ denote the set of
maximal subgroups $M$ of $G$ with $s \in M$ and $S \not\le M$. A \emph{novelty maximal} of $G$ is a
maximal subgroup $M$ such that $M \cap S$ is a proper non-maximal
subgroup of $S$.   We shall use ATLAS notation for the names of simple groups. For example, $\oo^+_8(3)$ denotes the simple $8$-dimensional orthogonal group of plus type defined over $\mathbb{F}_3$, while $\mathrm{GO}^+_8(3)$ is the corresponding general orthogonal group.

\begin{mainthm}
\label{thm:maxsubnoconj}
Let $S$ be a non-abelian finite simple group, and if $S \cong \oo^+_8(q)$, then suppose that $q \notin \{2,3,5\}$. Then $S$ contains
an element $s$ such that: 
\begin{enumerate}[label={(\roman*)},font=\upshape]
\item \label{maxsub1} $N_S(\langle s \rangle) > C_S(s)$; and
\item \label{maxsub2} for each almost simple group $G$ with socle $S$, there is at most one novelty maximal in $\mathcal{M}'_G(s)$, and no two subgroups in $\mathcal{M}'_G(s)$ are $G$-conjugate.
\end{enumerate}
\end{mainthm}

Even once we have proved that all finite groups satisfying the independence property are supersoluble, it is not straightforward to classify these groups. Indeed, the description of these groups is neither natural nor easy, as evidenced by the following statement.  Throughout, we denote the Frattini subgroup of a group $G$ by $\frat(G)$.


\begin{mainthm}\label{thm:supersol_struc}
	A finite group $G$ satisfies the independence property if and only if one of the following occurs:
	\begin{enumerate}[label={(\roman*)},font=\upshape]
		\item $G$ is a cyclic group of prime power order;
		\item $G$ is the quaternion group $Q_8$ of order 8; or
		\item \label{supersol3} $G\cong (V_1^{\delta_1} \times \dots \times
                  {V_r^{\delta_r})}\rtimes H$,  where  $H$ is abelian, $\delta_1, \dots, \delta_r$ are positive integers for some $r \ge 0$, 
                  $V_1,\dots,V_r$ are irreducible $H$-modules
				on which $H$ acts non-trivially, and the following
                  statements hold.  Here, for $h \in H$, we write $I_h:=\{ i\in \{1,\dots,r\} \mid h \in C_H(V_i)\}$.
		\begin{enumerate}
				\item If $\delta_i=1$, then $|H/C_H(V_i)|$ is prime.
				\item  If $|V_i| = |V_j|$, then $i = j$.
				\item $(|H|,|V_i|)=1$ for all $i \in
                                  \{1, \ldots, r\}.$
			
		\item 
                  For all $x, y \in H$, if $y \in\langle x \rangle \frat(H)$  and 
		$I_x \subseteq I_y$, then one of $x$ and $y$ is a power of the other. If in addition $I_x \neq \emptyset$, then $y \in \langle x \rangle$.
		\end{enumerate}
			\end{enumerate}
                      \end{mainthm}

In part (iii), notice that if $G$ is abelian, then $G = H$ and $r = 0$, hence (d) implies that $\frat(H) = 1$.

The structure of the paper is as follows. Theorems \ref{thm:noncomint}
and \ref{thm:maxsubnoconj} are proved in Section
\ref{sec:nonconjsub}. In Section \ref{riduco}, we prove that a
finite group $G$ satisfying the independence property is
supersoluble. The structure of the finite supersoluble groups
satisfying the independence property is investigated in Section
\ref{superind}, where Theorem~\ref{thm:supersol_struc} is proved.
Finally, finite groups with the rank-independence property are
studied in Section \ref{rankperf},  where we prove
Theorems~\ref{thm:rank_perfect2} and~\ref{thm:rank_perfect3}.

\section{Non-conjugate maximal subgroups of almost simple groups}
\label{sec:nonconjsub}

In this section, we prove Theorem \ref{thm:noncomint}. To do so, we shall first prove several elementary
lemmas and show how Theorem~\ref{thm:noncomint} follows from
Theorems~\ref{thm:maxsubnoconj} and \ref{thm:o8plusexceps}.
We then work through the families
of finite simple groups, proving Theorems~\ref{thm:maxsubnoconj} and~\ref{thm:o8plusexceps}: 
for most families, Theorem \ref{thm:maxsubnoconj} follows easily from known
results on elements of $S$ that lie in few maximal subgroups. However,
significantly more work is required in the case where $S$ is
orthogonal of plus type.

\subsection{Preliminary results, and statement of Theorem~\ref{thm:o8plusexceps}}

Recall the notation $\mathcal{M}_G(s)$ and
$\mathcal{M'}_G(s)$ from just before Theorems~\ref{thm:noncomint} and~\ref{thm:maxsubnoconj}, respectively. 

\begin{thm}
\label{thm:o8plusexceps}
Let $S := \oo^+_8(q)$, with $q \in \{2,3,5\}$. Then $S$ contains an element $s$ such that, for any almost simple group $G$ with socle $S$, the following statements hold.
\begin{enumerate}[label={(\roman*)},font=\upshape]
\item \label{o8plus1} $\bigcap_{M \in \mathcal{M}_G(s)} M$ contains an element of $N_S(\langle s \rangle) \setminus C_S(s)$.
\item If $G > S$, then there exists an  $\alpha_G \in \aut(S)$ such that $\mathcal{M}'_G(s^{\alpha_G})$ contains at most one novelty maximal, and no two subgroups in $\mathcal{M}'_G(s^{\alpha_G})$ are conjugate in $G$.
\end{enumerate}
On the other hand, for each $r \in S$, the set $\mathcal{M}_S(r)$
contains two $S$-conjugate
subgroups, and there exists a group 
$R$ with $S < R \le \aut(S)$ such that $\mathcal{M}'_R(r)$ contains two
$R$-conjugate subgroups. 
\end{thm}

This theorem suggests that the statements about the group
 $\Omega^+_8(5)$ from \cite[p.~767]{guralnickkantor} and \cite[Lemma
5.15(b)]{bgk} are not quite correct. We also note that Theorem
\ref{thm:maxsubnoconj} implies that, when $S$ is as in that theorem,
$G$-conjugate subgroups in $\mathcal{M}_G(s)$ correspond to
conjugate maximal subgroups of $G/S$, and vice versa.
On the other hand, the final part of Theorem~\ref{thm:o8plusexceps} shows that Theorem
\ref{thm:maxsubnoconj} does not hold when $S =
\oo^+_8(q)$ with $q \in \{2,3,5\}$. In particular, there
is no $s \in S$ such that we may set $\alpha_G = 1$ for all $G > S$. 

We now state several
elementary but useful results. The first of these follows readily from the Orbit-Stabiliser Theorem, and is well
known; for example, see \cite[Lemma 2.4]{guralnickkantor} and the
proof of \cite[Lemma 5.9]{bgk}. We will usually apply this
result in the special case where $H$ is a (non-normal) maximal subgroup of an
almost simple group $G$. 

\begin{lemma}
\label{lem:conjclassmax}
Let $H$ be a self-normalising subgroup of a finite group $J$, let $s
\in H$, and let $k$ be the number of $J$-conjugates of $H$ that
contain $s$. Then the following statements hold. 
\begin{enumerate}[label={(\roman*)},font=\upshape]
\item\label{conjclassmax1} $k \ge |N_J(\langle s \rangle):N_H(\langle
  s \rangle)|$, with equality if all $J$-conjugates of $\langle s
  \rangle$ in $H$ are $H$-conjugate. 
\item\label{conjclassmax2} Suppose that $|C_H(f)|$ is constant for all
  $f \in s^J \cap H$, and let $r$ be the number of $H$-classes of elements in $s^J \cap H$. Then $k = r|C_J(s):C_H(s)|$. 
\end{enumerate}
\end{lemma}

Our next result describes how novelty maximal subgroups of  a finite almost simple group $G$ may
arise as elements of $\mathcal{M}'_G(s)$,  for an element $s$ of the socle $S$ of $G$. Here, we adapt
  \cite[Lemma 2.4(2)]{bgk} and its proof, which 
assume that $|G:S|$ is prime, to the general case.

\begin{lemma}
\label{lem:conjmaxint}
Let $G$ be a finite almost simple group with socle $S$, and $s \in S \setminus \{1\}$. Then each subgroup in $\mathcal{M}'_G(s)$ is the normaliser in $G$
of the intersection of one or more $G$-conjugate subgroups in
$\mathcal{M}_S(s)$. 
In particular,
\[
  \bigcap_{M \in \mathcal{M}_S(s)} M \le \bigcap_{M \in
    \mathcal{M}_G(s)} M.
\]
\end{lemma}

\begin{proof}
Let $M \in \mathcal{M}'_G(s)$. Then $M \cap S \le L$ for some $L \in \mathcal{M}_S(s)$. Since $M \cap S \trianglelefteq M$, we observe that $s^{x^{-1}} \in M \cap S \le L$ for each $x \in M$, and so $s \in L^x$. Thus $X:=\bigcap_{x \in M} L^x$ is an intersection of $G$-conjugate subgroups in $\mathcal{M}_S(s)$. It is clear that $M \le N_G(X) < G$, and hence $M = N_G(X)$.
\end{proof}


Notice that, for a given non-abelian finite simple group $S$, the
final claim in Lemma~\ref{lem:conjmaxint} shows that Theorem
\ref{thm:noncomint} holds for each finite almost simple group $G$ with
socle $S$ if and only if it holds in the special case $G = S$. 

Next, observe from the Orbit-Stabiliser Theorem that if $L$ and $M$
are non-conjugate maximal subgroups in $S$ such that $N_G(L)$
  and $N_G(M)$
are maximal in $G$, then these normalisers are not conjugate
in $G$. Combining this with Lemma~\ref{lem:conjmaxint} yields
the
following lemma. 
 
\begin{lemma}
\label{lem:ordinarynonconj}
Let $G$ be a finite almost simple group with socle $S$, and let $s$ be an element of 
$S$ such that no two subgroups in $\mathcal{M}_S(s)$ are
$S$-conjugate. If the novelty maximals in $\mathcal{M}'_G(s)$ are pairwise
non-conjugate in $G$, then all subgroups in $\mathcal{M}'_G(s)$ are
pairwise non-conjugate in $G$. In particular, if $G > S$ and $\mathcal{M}_S(s)$
contains at most two $G$-conjugate subgroups, then
Theorem~\ref{thm:maxsubnoconj}(ii) holds for $G$ and $s$.
\end{lemma}

We finish this subsection by proving that Theorems~\ref{thm:maxsubnoconj}
and~\ref{thm:o8plusexceps} imply Theorem~\ref{thm:noncomint}. Note that we have included the statements about novelty maximals in the first two of these theorems for general interest; they are not required to prove Theorem~\ref{thm:noncomint}.

\begin{proof}[Proof of Theorem \ref{thm:noncomint} (assuming Theorems~\ref{thm:maxsubnoconj}
and~\ref{thm:o8plusexceps})]
If $S \cong \oo^+_8(q)$ with $q \in \{2,3,5\}$, then the
result is an immediate consequence of
Theorem~\ref{thm:o8plusexceps}\ref{o8plus1}.

In the remaining cases, let $s$ be the element whose existence
  is guaranteed  by Theorem~\ref{thm:maxsubnoconj}, so that there
  exists 
  a suitable $x \in N_S(\langle s
  \rangle) \setminus C_S(s)$.
 By  Theorem \ref{thm:maxsubnoconj}\ref{maxsub2}, no two subgroups
in $\mathcal{M}_S(s)$ are $S$-conjugate. Lemma
\ref{lem:conjclassmax}\ref{conjclassmax1}, with $J = S$ and $H = M \in
\mathcal{M}_S(s)$,  then shows that
$N_S(\langle s \rangle)= N_M(\langle s \rangle)$, and so
\[
  x \in N_S(\langle s \rangle) \le \bigcap_{M \in
  \mathcal{M}_S(s)} M.\]
Finally,   Lemma~\ref{lem:conjmaxint}
shows that for each almost simple group $G$ with socle $S$, the element
$x$ lies in $\bigcap_{M \in \mathcal{M}_G(s)} M$.
%
\end{proof}

The remainder of this section is dedicated to proving Theorems
\ref{thm:maxsubnoconj} and \ref{thm:o8plusexceps}. As stated in our
proofs, several of our 
arguments involve computations performed via GAP \cite{GAP} and Magma
\cite{magma}. Our proofs also contain details about specific
elements $s$ that may be chosen.

\subsection{Alternating, sporadic and exceptional groups}

In this subsection, we prove Theorem \ref{thm:maxsubnoconj} when $S$ is an alternating, sporadic or exceptional group.



\begin{prop}
\label{prop:altalmostsimple}
Theorem \ref{thm:maxsubnoconj} holds when $S$ is an alternating group $\alt_n$.
\end{prop}

\begin{proof} 
Suppose first that $n \in \{5, 6, 7, 9\}$, and let $s$ be any
element of $S$ of order $5$, $5$, $7$ or $15$, respectively. We deduce
from the ATLAS \cite{ATLAS} and calculations using the GAP Character Table Library \cite{GAPchar}
that: 
\begin{enumerate}[label={(\roman*)},font=\upshape]
\item the subgroups in $\mathcal{M}_S(s)$ lie in exactly one, two, two
  and three $S$-conjugacy classes of subgroups, respectively; and 
\item for each $M \in \mathcal{M}_S(s)$, all $S$-conjugates of
  $\langle s \rangle$ in $M$ are $M$-conjugate, and $N_M(\langle s
  \rangle) = N_S(\langle s \rangle) > C_S(s)$, so
    Theorem~\ref{thm:maxsubnoconj}\ref{maxsub1} holds for $s$.
\end{enumerate}
Using Lemma \ref{lem:conjclassmax}\ref{conjclassmax1}, with $J = S$
and $H \in \mathcal{M}_S(s)$, we see from (ii) that $k = 1$, and so no two
subgroups in $\mathcal{M}_S(s)$ are conjugate in
$S$. Moreover, if $|\mathcal{M}_S(s)| > 2$, so that $n = 9$, then
the three subgroups in $\mathcal{M}_S(s)$ are pairwise non-isomorphic,
and hence pairwise non-conjugate in $G$. Thus
Theorem~\ref{thm:maxsubnoconj}\ref{maxsub2}  holds for $S$ by
Lemma~\ref{lem:ordinarynonconj}.

Next, suppose that $n \ge 8$ is even. Additionally,
let \[s:=(1,\ldots,p)(p+1,\ldots,n) \in S,\] where $p$ is the largest
prime less than $n-1$. 
First observe that $s^2$ and $s^{-1}$ are both
  $\sym_n$-conjugate to $s$,
and so $|N_{\sym_n}(\langle s \rangle):\langle s \rangle| > 2$. Thus
$N_S(\langle s \rangle) > \langle s \rangle = C_S(s)$, and Theorem
\ref{thm:maxsubnoconj}\ref{maxsub1} holds for $S$. 
Next,
  Bertrand's Postulate  states that there exists a prime strictly between $i:=n/2 \ge 4$
  and $2i-2$, so 
$(p,n-p) = 1$. As observed in \cite{guralnicktracey}, it follows
that $|\mathcal{M}'_G(s)| = 1$. This is because the $p$-cycle
$s^{n-p}$ fixes at least $3$ points (since $n$ is even), and
so \cite{jordan} (see also 
\cite[Thm.~13.9]{wielandt}) implies that $\alt_n$ and $\sym_n$ are the
only primitive subgroups of $\sym_n$ containing $s^{n-p}$. It is also
clear that an imprimitive group preserving $j$ blocks of
  size $k$, 
with $jk = n$, cannot contain an element of order $p > n/2 \ge
j,k$. Therefore, 
$\mathcal{M}'_G(s) = \{(\sym_p \times \sym_{n-p}) \cap G\}$, and Theorem
\ref{thm:maxsubnoconj}\ref{maxsub2} holds for $S$.

Suppose now that $n \ge 11$ is odd. Here, we
let \[s:=(1,\ldots,p)(p+1,p+2)(p+3,\ldots,n) \in S,\] where $p$ is the
largest prime less than $n-5$. 
Applying Bertrand's Postulate to
$(n-1)/2-1$ yields $p > (n-1)/2-1$. Hence $p \ge (n-1)/2$, which is
greater than each proper divisor of $n$. In particular,  $n-(p+3)$ is coprime to each of $2$, $p$ and $n-(p+2)$, and
so 
$s^{n-(p+3)}$ and $s^{-1}$ are $\sym_n$-conjugate to $s$. Therefore, $N_S(\langle s \rangle) > C_S(s)$. We also observe, similarly to the previous
  case, that  $\alt_n$ is the only proper transitive subgroup of $\sym_n$
containing the $p$-cycle $s^{n-(p+2)}$. Thus $\mathcal{M}'_G(s)$ consists of
three intransitive subgroups, of shape $(\sym_{p} \times \sym_{n-p}) \cap
G$, $(\sym_{2} \times \sym_{n-2}) \cap G$ and $(\sym_{p+2} \times \sym_{n-(p+2)})
\cap G$, respectively. Since $(n-1)/2 \le p < n-5$, these
  subgroups are pairwise non-isomorphic, and hence
Theorem \ref{thm:maxsubnoconj} holds for $S$. 
\end{proof}

We shall treat the Tits group ${}^2\mathrm{F}_4(2)'$ as an exceptional group.

\begin{prop}
\label{prop:sporadalmostsimple}
Theorem \ref{thm:maxsubnoconj} holds when $S$ is a sporadic group.
\end{prop}

\begin{proof}
For each sporadic group $S$,
\cite[Table 1]{burnessharper} describes, for at least one element $g \in S$,  the
corresponding set $\mathcal{M}_S(g)$ (much of this information is also given
in \cite[Table IV]{guralnickkantor}). Suppose first that $S \notin \{\mathrm{M}_{12},\mathrm{Suz}\}$, and let $s$ be an element of $S$ described in \cite[Table 1]{burnessharper}, with $|s| = 17$ if $S = \mathrm{He}$, and $|s| = 22$ if $S = \mathrm{Fi}_{22}$.  We see in \cite[Table 1]{burnessharper} that if 
$\mathcal{M}_S(s)$ contains two isomorphic subgroups, then
$|\mathcal{M}_S(s)| = 2$ and $\langle s \rangle$ is a Sylow subgroup
of $S$. Since at least one maximal subgroup in $\mathcal{M}_S(s)$
contains $N_S(\langle s \rangle)$, and $\langle s \rangle$ is a Sylow subgroup, it follows from Lemma
\ref{lem:conjclassmax}\ref{conjclassmax1} that if
$\mathcal{M}_S(s)$ contains two isomorphic subgroups, then they are
not $S$-conjugate.
 
If instead $S = \mathrm{M}_{12}$, then let $s$ be any element of $S$
of order $11$, and if $S = \mathrm{Suz}$, then let $s$ be any element
of order $21$. Lemma \ref{lem:conjclassmax}\ref{conjclassmax1} and
character table calculations in GAP show that $\mathcal{M}_S(s)$
contains exactly three subgroups, which are pairwise
non-conjugate in $S$. 
Furthermore, precisely two of these subgroups are isomorphic
when $S = \mathrm{M}_{12}$, and they are pairwise
  non-isomorphic when $S =
\mathrm{Suz}$, so at most two are $G$-conjugate in each case.

For all $S$, Lemma~\ref{lem:ordinarynonconj} now yields Theorem~\ref{thm:maxsubnoconj}\ref{maxsub2}, and we deduce Theorem~\ref{thm:maxsubnoconj}\ref{maxsub1} from the ATLAS.
\end{proof}

\begin{prop}
\label{prop:excepalmostsimple}
Theorem \ref{thm:maxsubnoconj} holds when $S$ is an exceptional group.
\end{prop}

\begin{proof}
  Suppose first that $S$ is isomorphic to $\mathrm{G}_2(3)$, $\mathrm{G}_2(4)$, $\mathrm{F}_4(2)$
or ${}^2 \mathrm{F}_4(2)'$. Then \cite[p.~566]{burnessharper} describes an
element $s \in S$ of order $13$, $21$, $17$ or $16$, respectively,
that lies in exactly three, one, two or two maximal subgroups of $S$,
respectively. In each of these cases, we deduce from the ATLAS  that
$N_S(\langle s \rangle) > C_S(s)$, and from Lemma
\ref{lem:conjclassmax}\ref{conjclassmax1} and character table
calculations in GAP that no two maximal subgroups in
$\mathcal{M}_S(s)$ are conjugate in $S$. Additionally, if $S =
\mathrm{G}_2(3)$, then we see in the ATLAS that $\mathcal{M}'_G(s)$
contains no novelty maximal.
Thus Lemma \ref{lem:ordinarynonconj} implies that Theorem \ref{thm:maxsubnoconj} holds for $S$.

For each remaining finite exceptional simple 
$S$, Propositions
6.1 and 6.2 of \cite{guralnickkantor}, building on the work of
Weigel \cite[\S4]{weigel}, give a semisimple 
$s \in S$ such
that $|\mathcal{M}'_G(s)| \le 2$. In fact, if $|\mathcal{M}'_G(s)| =
2$, then $S$ is equal to  $\mathrm{G}_2(3^e)$ or $\mathrm{F}_4(2^e)$ with $e \ge 2$, 
$|\mathcal{M}_S(s)| = 2$, and $\mathcal{M}'_G(s)$ contains at
most one novelty maximal. Moreover, the two groups in
$\mathcal{M}_S(s)$ are not $S$-conjugate 
\cite[pp.~74--78]{weigel}. Note that when $S = \mathrm{G}_2(3^e)$, Weigel states only that the groups in $\mathcal{M}_S(s)$
are members of
two $S$-conjugacy classes, with each
group having shape $\SU_3(q).2$. We can use Lemma~\ref{lem:conjclassmax}\ref{conjclassmax1} and \cite[Tables 8.5, 8.6 \& 8.42]{bhrd} to show that the element $s$ of order $q^2-q+1$ lies in exactly one member of each of these conjugacy classes. 
Thus in all cases, Theorem~\ref{thm:maxsubnoconj}\ref{maxsub2} follows from
Lemma \ref{lem:ordinarynonconj}.

It remains to show that $N_S(\langle s \rangle) > C_S(s)$ in each case.
If $S$ is not isomorphic to
$\mathrm{E}_7(2)$ or $\mathrm{G}_2(q)$, then this is clear from \cite[Table
6]{guralnickmalle}. If instead $S \in \{\mathrm{E}_7(2), \mathrm{G}_2(q)\}$, then
$\langle s \rangle$ is a maximal torus of $S$ \cite[Tables 3,
  7 \&  10, p.~46]{carterweyl}. As $S$ is the set of fixed points under
a Frobenius endomorphism of a simply connected algebraic group,
it follows from
\cite[pp.~2011--2012]{deriziotis81} that $C_S(s) = \langle s
\rangle$. No non-abelian finite simple group contains a self-normalising cyclic
subgroup \cite{guangxiang}, and so $N_S(\langle s \rangle) > C_S(s)$. 
\end{proof}

\subsection{Classical groups}

This subsection consists of the proof of Theorem
\ref{thm:maxsubnoconj} when $S$ is classical, as well as the
proof of Theorem \ref{thm:o8plusexceps}. Let $u := 2$ if
$S$ is unitary and $u := 1$ otherwise, and let the natural module for
$S$ be $\FF_{q^u}^n$. We also let $\tilde S$ be a classical
quasisimple subgroup of $\gl_n(q^u)$ such that $\tilde S/Z(\tilde S)
\cong S$. For each $h \in S$ and $H \le S$, we will write $\tilde h$
to denote a preimage of $h$ in $S$, and $\tilde H$ to denote the
preimage of $H$ containing $Z(\tilde S)$.

If a classical subgroup $\tilde H$ of $\gl_n(q^u)$ contains irreducible cyclic subgroups,
then a generator $\tilde y$ of any such subgroup of
maximal order is 
a \emph{Singer cycle} of $H$, and the
\emph{Singer subgroup} $\langle \tilde y \rangle$ is equal to $\tilde
H \cap \langle \tilde g \rangle$ for some Singer cycle $\tilde g$ of
$\gl_n(q^u)$ (see \cite{bereczky,hestenes,huppert}). We will
also say that $y$ is a \emph{Singer cycle} of $H$. 

In order to prove Theorem \ref{thm:maxsubnoconj} when $S$ is a
classical group, we will often define a suitable $s$
via Singer cycles of subgroups of $S$. The following lemma details
important properties of such a Singer cycle. 

\begin{lemma}
\label{lem:singerprops}
Suppose that $S$ is 
listed in Table \ref{table:singerords}, and let $\tilde s$ be a Singer cycle of $\tilde S$. Then $|\tilde s|$ is given in the table. Moreover, $N_{\tilde S}(\langle \tilde s \rangle) > C_{\tilde S}(\tilde s)$ and $N_S(\langle s \rangle) > C_S(s)$.
\end{lemma}

\begin{proof}
The order of $\tilde s$ is given in \cite[Table 1]{bereczky}, and
\cite[p.~615]{praeger} yields $N_{\tilde S}(\langle \tilde s \rangle)
> C_{\tilde S}(\tilde s)$.  Hence either $N_S(\langle s \rangle) > C_S(s)$, or $\tilde s$ and $\tilde z \tilde s$ are $\tilde S$-conjugate for some non-identity scalar matrix $\tilde z \in \tilde S$. However, $\tilde z \tilde s$ and $\tilde s$ have
different $\mathbb{F}_{q^{un}}$-eigenvalues
on $\mathbb{F}_{q^u}^n$, and so $N_S(\langle s \rangle) > C_S(s)$. 
\end{proof}

\begin{table}
\centering
\renewcommand{\arraystretch}{1.1}
\caption{The order of a Singer cycle $\tilde s \in \tilde S$, for certain
  $S$.}
\label{table:singerords}
\begin{tabular}{ |c|c| }
\hline
$S$ & $|\tilde s|$ \\
\hline
\hline
$\psl_n(q)$ & $(q^n-1)/(q-1)$ \\
\hline
$\psu_n(q)$, $n$ odd & $(q^n+1)/(q+1)$\\
\hline
$\psp_n(q)$ & $q^{n/2}+1$\\
\hline
$\oo^-_n(q)$ & $(q^{n/2}+1)/(2,q-1)$\\
\hline
\end{tabular}
\end{table}

Note that each finite simple classical group that contains a
Singer cycle (i.e., that contains an irreducible cyclic subgroup)
appears in Table \ref{table:singerords} (see \cite[p.~188]{bereczky}).

The arguments later in this subsection do not apply to certain small classical groups. We therefore deal with these groups separately. 

\begin{prop}
\label{prop:smallclas}
Theorem \ref{thm:maxsubnoconj} holds for each classical group $S$ given in Table \ref{table:smallclassical}.
\end{prop}
\begin{table}
\centering
\renewcommand{\arraystretch}{1.1}
\caption{The size of $\mathcal{M}_S(s)$, where $S$ is a given classical simple group and $s$ is any element of $S$ of the given order.}
\label{table:smallclassical} \subfloat{
\begin{tabular}{ |c|c|c| }
\hline
$S$ & $|s|$ & $|\mathcal{M}_S(s)|$ \\
\hline
\hline
$\psl_3(4)$ & $7$ & $3$\\
\hline
$\psl_6(2)$ & $63$ & $2$\\
\hline
$\psu_3(3)$ & $7$ & $1$\\
\hline
$\psu_3(5)$ & $10$ & $2$\\
\hline
$\psu_4(2)$ & $9$ & $2$\\
\hline
$\psu_4(3)$ & $7$ & $7$\\
\hline
\end{tabular} }
\subfloat{ \begin{tabular}{ |c|c|c| }
\hline
$S$ & $|s|$ & $|\mathcal{M}_S(s)|$ \\
\hline
\hline
$\psu_5(2)$ & $11$ & $1$\\
\hline
$\psu_6(2)$ & $11$ & $4$\\
\hline
$\psp_6(2)$ & $15$ & $2$\\
\hline
$\psp_8(2)$ & $17$ & $3$\\
\hline
$\oo_7(3)$ & $14$ & $3$\\
\hline
$\oo^+_8(4)$ & $65$ & $3$\\
\hline
\end{tabular}}
\end{table}
\begin{proof}
  Let $s$ be any element of $S$ such that $|s|$ is as in Table \ref{table:smallclassical}. We can use Lemma \ref{lem:conjclassmax}\ref{conjclassmax1}, together with computations in Magma when $S = \oo^+_8(4)$, or character table calculations in GAP in the remaining cases, to determine the set $\mathcal{M}_S(s)$ (in certain cases, $\mathcal{M}_S(s)$, or its size, is also given in \cite[\S4]{bgk}, the proof of \cite[Prop.~6.3]{guralnickkantor}, or the proof of \cite[Thm.~6.1]{burnessharper}). In particular, $|\mathcal{M}_S(s)|$ has the value given in Table \ref{table:smallclassical}, and no two subgroups in $\mathcal{M}_S(s)$ are $S$-conjugate.

To prove Part (ii), we first note that if $S = \psp_8(2)$, then no two subgroups in $\mathcal{M}_S(s)$ are isomorphic, and if $S = \oo_7(3)$, then the three subgroups in $\mathcal{M}_S(s)$ lie in two isomorphism classes. If instead $S \in \{\psl_3(4),\oo^+_8(4)\}$, then we can show using Magma that the intersection of any two subgroups in $\mathcal{M}_S(s)$ is equal to the intersection of all three. When $|\mathcal{M}_S(s)| \le 3$, it now follows from Lemma \ref{lem:conjmaxint}  that $\mathcal{M}'_G(s)$ contains at most one novelty maximal. In addition, we observe from the ATLAS  that if $|\mathcal{M}_S(s)| > 3$ (that is, if $S \in \{\psu_4(3), \psu_6(2)\}$), then $\mathcal{M}'_G(s)$ contains no novelty maximals. Thus, in each case, Lemma \ref{lem:ordinarynonconj} implies that no two subgroups in $\mathcal{M}'_G(s)$ are $G$-conjugate.

For Part (i), we deduce that $N_S(\langle s \rangle) > C_S(s)$, using Lemma \ref{lem:singerprops} when $s$ is a Singer cycle of $S$, using Magma when $S = \oo^+_8(4) \cong \Omega^+_8(4)$, and using the ATLAS in the remaining cases. 
\end{proof}

In order to prove that Theorem \ref{thm:maxsubnoconj} holds in the remaining cases where $S$ contains a Singer cycle, we require the following elementary observation. Throughout the rest of this section, we shall abbreviate \emph{primitive prime divisor} to \emph{ppd}.

\begin{lemma}
\label{lem:ppdfieldaut}
Let $p$ be a prime, $f$ a positive integer, and $r$ a 
ppd of $p^f-1$. Then $r$ does not divide $f$.
\end{lemma}

\begin{proof}
Suppose, for a contradiction, that $r \mid f$. Then there exists a positive integer $k < f$ such that $p^f = (p^k)^r \equiv p^k \pmod r$. Thus $p^f-1 \equiv p^k-1 \pmod r$. As $r$ divides $p^f-1$, it therefore also divides $p^k-1$. This contradicts the primitivity of $r$ as a prime divisor of $p^f-1$.
\end{proof}

\begin{prop}
\label{prop:singercases}
Theorem \ref{thm:maxsubnoconj} holds when $S$ is linear, unitary of odd dimension, symplectic, or orthogonal of minus type with $n \ge 8$.
\end{prop}

\begin{proof}
  We may assume that $S$ has not been 
dealt with in
  Proposition \ref{prop:smallclas}, and by Proposition \ref{prop:altalmostsimple} that $S$ is not isomorphic to an alternating group. We will also assume  that $n \ge 4$ in the symplectic case, 
and that $q \ne 7$ if $n = 2$; we will treat $\psl_2(7)$ as $\psl_3(2)$.

Let $s$ be a Singer cycle of $S$.
Then $N_S(\langle s \rangle) > C_S(s)$ by Lemma \ref{lem:singerprops}. Additionally,  the main theorem of \cite{bereczky} shows that if $\mathcal{M}_S(s)$ contains a maximal subgroup that is not an extension field type subgroup, then $S = \psp_n(q)$ with $q$ even. Moreover, if there is more than one such maximal subgroup, then there are exactly two, they are not conjugate in $S$, and $n = 4$ \cite[Thm.~1.1, p.~94]{mallesaxlweigel}.

If $S = \psl_2(q)$, then our exclusions on $q$ imply that 
$\mathcal{M}_S(s) = \{N_S(\langle s \rangle)\}$. Otherwise, we observe
from Lemma \ref{lem:singerprops} and Zsigmondy's Theorem (since
we exclude the groups in Proposition~\ref{prop:smallclas}) that $|s|$ is divisible by a 
ppd $r$ of $q^{un}-1$.
In addition, by Lemma \ref{lem:ppdfieldaut}, $r \nmid un$. Hence \cite[Lemma 2.12]{bgk} implies that no two extension field subgroups in $\mathcal{M}_S(s)$ are isomorphic.
%
The result now follows from Lemma~\ref{lem:ordinarynonconj}.
\end{proof}

\begin{prop}
Theorem \ref{thm:maxsubnoconj} holds when $S$ is unitary of even dimension or orthogonal of odd dimension.
\end{prop}

\begin{proof}
  We may assume that $S$ has not been 
  dealt with in Proposition~\ref{prop:smallclas}, and by Proposition \ref{prop:singercases} that $n \ge 4$ in the unitary case and $n \ge 7$ otherwise. 

Let
 $\tilde M$ be the stabiliser in $\tilde S$ of a non-degenerate $(n-1)$-dimensional
subspace of $\mathbb{F}_{q^u}^n$, of minus type if $S$ is
orthogonal. Then $\tilde M$ contains a subgroup  $H$ isomorphic to
$\SU_{n-1}(q)$ or $\Omega^-_{n-1}(q)$. Let $\tilde{s}$ be a Singer cycle of $H$ (with respect to its action on $\mathbb{F}_{q^u}^{n-1}$), and let $\tilde K:=Z(\tilde S)H$. Then Lemma \ref{lem:singerprops} (or an easy generalisation of this result in the case $Z(\tilde S) \not\le H$) yields $N_K(\langle s \rangle) > C_K(s)$, and hence $N_S(\langle s \rangle) > C_S(s)$. Furthermore, $\mathcal{M}_S(s) = \{M\}$ by \cite[Thm.~1.1, p.~93]{mallesaxlweigel}, and so the result follows by Lemma~\ref{lem:ordinarynonconj}.
\end{proof}

The remainder of this section concerns the case $S = \oo^+_{2m}(q)$, where without loss of generality $m \ge 4$, so that $\tilde S = \Omega^+_{2m}(q)$.
The case $m$ odd is straightforward.

\begin{prop}
Theorem \ref{thm:maxsubnoconj} holds when $S = \oo^+_{2m}(q)$, with $m \ge 5$ odd.
\end{prop}

\begin{proof}
  Choose $\tilde s$ to be the element of order
  $(q^{(m-1)/2}+1)(q^{(m+1)/2}+1)/(4,q-1)$ of $\tilde S$ described in Proposition
  5.13 of \cite{bgk}, so that $\tilde s$ is a product of two Singer cycles of orthogonal subgroups of $\tilde S$, unless $q \equiv 3 \pmod 4$. Then \cite[Proposition 5.13]{bgk} shows that $\tilde s$ lies in a
  unique maximal subgroup of $\tilde S$, namely, the stabiliser of an
$(m-1)$-dimensional subspace of $\mathbb{F}_q^{2m}$ of minus type. Hence $|\mathcal{M}_S(s)| = 1$, and Part (ii) follows from Lemma~\ref{lem:ordinarynonconj}. 

It remains to show that $N_S(\langle s \rangle) > C_S(s)$. The set of $\overline{\FF_q}$-eigenvalues of $\tilde s$ is a subset of $\mathbb{F}_{q^{(m-1)(m+1)/2}}$
that is closed under conjugation by the field automorphism $\alpha \mapsto
\alpha^q$. Thus the distinct elements $\tilde s$, $\tilde s^q$ and
$\tilde s^{q^2}$ all have the same eigenvalues. Note also that
$\langle \tilde s \rangle$ stabilises no one-dimensional subspace of $\mathbb{F}_q^{2m}$; otherwise, $\mathcal{M}_S(s)$ would contain the stabiliser of such a subspace. Thus it follows from \cite[pp.~38--39]{wall} (see also \cite[Prop.~2.11]{bgk} for an explicit statement) and \cite[Thms 6.1.12 \& 6.1.15]{defranceschi} that $\tilde s$ is $\tilde S$-conjugate to at least one of $\tilde s^q$ and $\tilde s^{q^2}$. As $\tilde s$ and $-\tilde s$ do not have equal eigenvalues when $q$ is odd, we deduce that $N_S(\langle s \rangle) > C_S(s)$.
%
\end{proof}

The case where $m$ is even is much more involved.

\begin{proof}[Proof of Theorem \ref{thm:o8plusexceps}]
  As specified below, many of the facts about 
    almost simple groups mentioned in this proof were deduced using Magma. In the case $q = 5$, most of our computations were performed in the matrix group $\tilde S = \Omega^+_8(5)$, since this group's maximal subgroups can be constructed using the \texttt{ClassicalMaximals} function.
  In general, in order to show that an element $r \in S$ lies in multiple $G$-conjugates of a given maximal subgroup $M \in \mathcal{M}'_G(r)$, it suffices by Lemma \ref{lem:conjclassmax}\ref{conjclassmax1} to verify that $N_G(\langle r \rangle) > N_M(\langle r \rangle)$.

  The majority of the proof will be divided into two cases, depending on 
 $q$. However, we first observe computationally that for the element $s$ specified in each case below, $\bigcap_{M \in \mathcal{M}_S(s)} M$ contains an element of $N_S(\langle s \rangle) \setminus C_S(s)$, and hence by Lemma~\ref{lem:conjmaxint}, so does $\bigcap_{M \in \mathcal{M}_G(s)} M$. Computations also show that for each $r \in S$, there exist two $S$-conjugate subgroups in $\mathcal{M}_S(r)$. Therefore it remains to prove Part (ii), and the rest of the final claim of the theorem.

\smallskip

\noindent \textbf{Case (a)}: $q = 2$. Let $s$ be any element of $S$ of order $15$. If $|G:S| \ge 3$, then we observe from the ATLAS  that, up to $G$-conjugacy, $G$ has a unique maximal subgroup of order divisible by $|s|$ that does not contain $S$. In fact, we see using Magma that this maximal subgroup has no element of order $|s|$, and so $\mathcal{M}'_G(s) = \varnothing$.

If instead $|G:S| = 2$, so that $G \cong \mathrm{SO}^+_8(2)$, then no novelty maximal of $G$ has order divisible by $|s|$ . Furthermore, computations show that there exist $\alpha,\beta \in \mathrm{Aut}(S)$ such that $\mathcal{M}'_G(s^\alpha) = \varnothing$, while two subgroups in $\mathcal{M}'_G(s^\beta)$ are conjugate in $G$. We also observe via Magma that for each $h \in S$ that is not $G$-conjugate to $s^\alpha$, there exist two $G$-conjugate subgroups in $\mathcal{M}'_G(h)$. Hence for each $r \in S$, there exists an $\mathrm{Aut}(S)$-conjugate $R$ of $G$ such that $\mathcal{M}'_R(r)$ contains two $R$-conjugate subgroups.

\smallskip

\noindent \textbf{Case (b)}: $q \in \{3,5\}$. Magma computations show that $S$ has precisely three conjugacy classes of cyclic subgroups of order $(q^3+1)/2$, and for any element $s$ of this order, the subgroups in $\mathcal{M}_S(s)$ isomorphic to $\Omega_7(q)$ are members of two $S$-conjugacy classes (with these classes depending on the class of $\langle s \rangle$). Moreover, the maximal subgroups of $S$ isomorphic
to $\Omega_7(q)$ fall into six $S$-conjugacy classes, and
only one $\mathrm{Aut}(S)$-conjugacy class \cite[Table
I]{kleidmano8}. We therefore deduce that there exist $\theta_1,\theta_2 \in \mathrm{Aut}(S)$ such that $\langle s \rangle$, $\langle s^{\theta_1} \rangle$ and $\langle s^{\theta_2} \rangle$ lie in three distinct $S$-classes.
Additional computations show that for each $t \in \{s,s^{\theta_1},s^{\theta_2}\}$, the set $\mathcal{M}_S(t)$ contains precisely three subgroups that are not isomorphic to $\Omega_7(q)$: the stabiliser $K_{t,1}$ in $S$ of a 
$6$-dimensional subspace of $\mathbb{F}_q^8$ of minus type, and two of its images $K_{t,2}$ and $K_{t,3}$ under triality automorphisms of $S$ (see \cite[Table I]{kleidmano8}). In particular, these three subgroups are pairwise non-conjugate in $S$.


Next, we observe using \cite[Table 8.50]{bhrd} that if $M$ is a
novelty maximal subgroup of $G$ whose order is divisible by $|s|$,
then either $M \cap S \cong \mathrm{G}_2(q)$; $M \cap S$ is an extension of a
$2$-group by $\alt_8$ or by $\psl_3(2)$; or $M \cap S$ is
$S$-conjugate to $U_t:=K_{t,1} \cap K_{t,2} \cap K_{t,3}$. However, we
see in the ATLAS that $\mathrm{G}_2(q)$ contains no element of order $|s|$, and no maximal subgroup in $\mathcal{M}_S(t)$ has order divisible by the order of any of these 
extensions of $2$-groups. It therefore follows from Lemma \ref{lem:conjmaxint} that any novelty maximal subgroup in $\mathcal{M}'_G(t)$ is equal to $N_G(U_t)$.

We now conclude from \cite[Table I]{kleidmano8} that, for each $G > S$, there exists $\alpha_G \in \{1,\theta_1,\theta_2\}$ such that,  when $t =s^{\alpha_G}$, any subgroup in $\mathcal{M}'_G(t)$ is equal to either $N_G(U_t)$ or $N_G(K_{t,i})$ for some $i \in \{1,2,3\}$. Moreover, since $K_{t,i}$ and $K_{t,j}$ are not $S$-conjugate when $i \ne j$, if $N_G(K_{t,i})$ and $N_G(K_{t,j})$ are maximal in $G$ then they are not $G$-conjugate. Therefore, no two subgroups in $\mathcal{M}'_G(t)$ are $G$-conjugate.

Finally, let $r \in S$, and suppose that $G \cong \mathrm{PSO}^+_8(q)$ and that no two subgroups in $\mathcal{M}'_G(r)$ are $G$-conjugate. Using Magma and \cite[Table I]{kleidmano8}, we deduce that either there exists $\rho \in \mathrm{Aut}(S)$ such that $\mathcal{M}'_{G^\rho}(r)$ contains two $G^\rho$-conjugate subgroups; or $q = 5$, and there exists a graph automorphism $\pi$ of $S$ (of order $2$ or $3$) such that $\mathcal{M}'_{\langle S, \pi \rangle}(r)$ contains two $\langle S,\pi \rangle$-conjugate subgroups. Therefore, for each $r \in S$, there exists a group $R$ with $S < R \le \mathrm{Aut}(S)$ such that $\mathcal{M}'_R(r)$ contains two $R$-conjugate subgroups.
\end{proof}

Now, to prove Theorem \ref{thm:maxsubnoconj} 
for $m$ even, we would like to use Proposition~5.14 and Lemma~5.15 of \cite{bgk}. However, we require a slightly stronger statement than Proposition~5.14. Additionally, the statement of Proposition~5.14 is not quite correct when $4 \mid m$, nor is the proof of Lemma~5.15(b)--(c).
We therefore prove the following result, much of whose proof is similar to 
\cite[\S5]{bgk}, and which implies that \cite[Lemma 5.15(b)]{bgk} is in fact correct when $q \ge 7$.  We exclude the groups $\oo^+_8(q)$ for $q \le 5$, which we dealt with in Theorem~\ref{thm:o8plusexceps} and Proposition \ref{prop:smallclas}.

\begin{thm}
\label{thm:orthogplusmaxes}
Suppose that $S = \oo^+_{2m}(q)$, with $m \ge 4$ even, and $q \ge 7$ if $m = 4$. Then $\tilde S$ contains an element $\tilde s$ of order $k = k(m, q)$, where
\[k := \begin{cases}
(q^2+1)/(2,q-1), & \text{ if } m = 4, \\ (q^{(m-2)/2}+1)(q^{(m+2)/2}+1)/(4,q-1), & \text{ if } 4 \nmid m,\\
(q^{(m-2)/2}+1)(q^{(m+2)/2}+1)/(q+1), & \text{ if } m = 8 \text{ and } q = 2,\\
(q^{(m-2)/2}+1)(q^{(m+2)/2}+1)/(q+1)^2, & \text{ otherwise},
\end{cases}\]
such that the following statements hold.
\begin{enumerate}[label={(\roman*)},font=\upshape]
\item There are precisely three subgroups in $\mathcal{M}_{\tilde S}(\tilde s)$, which we will denote by $\tilde K_1$, $\tilde K_2$ and $\tilde L$.
\item The groups $\tilde K_1$ and $\tilde K_2$ are extension field subgroups that are
  not conjugate in $\tilde S$. If $4 < m \equiv 0 \pmod 4$, then $\tilde K_1$ and $\tilde K_2$ are of (Aschbacher) type $\mathrm{GU}_m(q)$. Otherwise, they are of type $\mathrm{GO}^+_m(q^2)$.
\item Let $V:=\mathbb{F}_q^{2m}$. If $m = 4$, then $\tilde L$ is imprimitive, stabilising the decomposition $V = U \perp U^\perp$, for some 
four-dimensional subspace $U$ of $V$ of minus type. Otherwise, $\tilde L$ is the stabiliser of
an  $(m-2)$-dimensional subspace of $V$ of minus type. 
\end{enumerate}
\end{thm}

Note that the case $m = 8$ and $q = 2$ is exceptional as here $q^{m-2}-1 = 2^6-1$ has no ppd. We begin by defining the element $\tilde s$ that we will use in each case.

\begin{assump}
\label{assump:selt}
Suppose that $S = \oo^+_{2m}(q)$, with $m \ge 4$ even, and $q \ge 7$ if $m = 4$. Additionally, let $k = k(m,q)$ be as in Theorem~\ref{thm:orthogplusmaxes}. If $4 \nmid m$, then let $\tilde s \in \tilde S$ be as described in \cite[Prop.~5.14]{bgk}, so that $\langle \tilde s\rangle$ has order $k$
and 
stabilises exactly two proper non-zero subspaces of  $V:= \mathbb{F}_q^{2m}$, namely, 
an $(m-2)$-dimensional subspace of $V$ of minus type and its orthogonal complement.

From now on, suppose that $4 \mid m$. Let $\tilde X$ be the stabiliser in $\tilde S$ of a 
subspace $U$ of $V$ of minus type, of dimension four if $m = 4$ or dimension $m-2$ if $m > 4$. Then $U^\perp$ is also 
of minus type. Additionally, $\tilde X$ contains the subgroup $R_1 \times R_2$, where $R_1 \cong \Omega(U)$
and $R_2 \cong \Omega(U^\perp)$. 

We define $\tilde r_1$ and $\tilde r_2$ to be elements of Singer subgroups of $R_1$ and $R_2$, respectively (corresponding to their actions on $U$ and $U^\perp$), as follows; in each case, an element of the specified order exists by Lemma \ref{lem:singerprops}. If $m = 4$, then let $\tilde r_1$ and $\tilde r_2$ both have order $k$. Then $\tilde r_1$ has eigenvalues $\alpha$, $\alpha^q$, $\alpha^{q^2}$ and $\alpha^{q^3}$ on $U$, while $\tilde r_2$ has eigenvalues $\beta$, $\beta^q$, $\beta^{q^2}$ and $\beta^{q^3}$ on $U^\perp$, where $\alpha$ and $\beta$ are elements of $\mathbb{F}_{q^4}$ of order $k$. 
Similarly to the proof of \cite[Lemma 5.15]{bgk}, we choose $\tilde r_2$ in this case so that these eight eigenvalues are all distinct (we will place further restrictions on $\tilde r_2$ in the case $m = 4$ in the proof of Proposition \ref{prop:omega8subgps}). 
If instead $m > 4$, then let $\tilde r_1$ and $\tilde r_2$ have order $(q^{(m-2)/2}+1)/t$ and $(q^{(m+2)/2}+1)/(q+1)$, respectively, where $t := 1$ if $m = 8$ and $q = 2$, and $t := q+1$ in all remaining cases. For all $m$ (with $4 \mid m$), let $\tilde s:=(\tilde r_1,\tilde r_2) \in \tilde X$. We deduce 
  that $|\tilde s| = k$.
\end{assump}

Next, we show that when $4 \mid m$, the group $\tilde L$ from Theorem \ref{thm:orthogplusmaxes} is the unique subgroup in $\mathcal{M}_{\tilde S}(\tilde s)$ that is not isomorphic to the specified extension field group $\tilde K_1$. We first state two preliminary results.

\begin{lemma}
\label{lem:orthplussubspaces}
Let $\tilde S$, $U$ and $\tilde s$ be as in Assumption \ref{assump:selt}, with $4 \mid m$. Then $U$ and $U^\perp$ are the only proper non-zero subspaces of $\mathbb{F}_q^n$ stabilised by $\langle \tilde s \rangle$.
\end{lemma}

\begin{proof}
If $m = 8$ and $q = 2$, then the element $\tilde r_1$ from Assumption \ref{assump:selt} is a Singer cycle of $\Omega^-_6(2)$ by Lemma \ref{lem:singerprops}, and hence $\langle \tilde r_1 \rangle$ acts irreducibly on $U$. In the remaining cases, $|\tilde r_1|$ is divisible by each ppd of $q^{\dim(U)}-1$, and in all cases, $|\tilde r_2|$ is divisible by each ppd of $q^{\dim(U^\perp)}-1$. We therefore deduce from \cite[Thm.~3.5]{hering} that $\tilde s = (\tilde r_1,\tilde r_2)$ acts irreducibly on each of $U$ and $U^\perp$.

Now, if $m = 4$, then the eigenvalues of $\tilde s$ on $U$ are distinct from its eigenvalues on $U^\perp$, and otherwise, $\dim(U) \ne \dim(U^\perp)$. Thus $U$ and $U^\perp$ are non-isomorphic irreducible $\mathbb{F}_q[\langle \tilde s \rangle]$-modules, and the result follows.
\end{proof}

\begin{lemma}
\label{lem:symelt}
Let $m \ge 8$ be a multiple of $4$, and suppose that $m+3$ is a ppd of $q^{m+2}-1$. Then the symmetric group $\sym_{2m+2}$ contains no element of order  $k=k(m,q)$, as defined in Theorem~\ref{thm:orthogplusmaxes}.
\end{lemma}

\begin{proof}
  If $q = 2$ and $m = 8$, then $\sym_{2m+2} = \sym_{18}$ contains no element of order $k = 99$. Assume therefore that $q > 2$ or $m > 8$, and suppose for a contradiction that $y$ is an element of $\sym_{2m+2}$ of order $k$. Observe that $k$ is divisible by each ppd of $q^{m+2}-1$, and hence by $m+3$. Similarly, $k$ is divisible by each ppd of $q^{m-2}-1$, and such a prime is no smaller than $m-1$ (see \cite[Remark 1.1]{gpps}). As $(m-1)+(m+3)$ is the degree of $\sym_{2m+2}$, we deduce that $k$ is the product $(m-1)(m+3)$ of two primes. 
  Now, $(q^{(m-2)/2}+1,q^{(m+2)/2}+1) = q+1$, and it follows that $m-1 = (q^{(m-2)/2}+1)/(q+1)$ and $m+3 = (q^{(m+2)/2}+1)/(q+1)$. Substituting these equalities into $(m-1) + 4 = m+3$, we deduce that  $4 = q^{(m-2)/2}(q-1)$, a contradiction for all $q$ and $m$.
\end{proof}

\begin{prop}
\label{prop:omegamultsubgps}
Let $\tilde S$, $U$ and $\tilde s$ be as in Assumption \ref{assump:selt}, with $m \ge 8$ and $4 \mid m$.
Then the stabiliser of $U$ in $\tilde S$ is the unique subgroup in $\mathcal{M}_{\tilde S}(\tilde s)$ that is not an extension field subgroup of type $\mathrm{GU}_m(q)$.
\end{prop}

\begin{proof}
Notice that $|\tilde s|$ is divisible by each ppd of $q^{m+2}-1$. The main theorem of \cite{gpps} yields all possibilities for maximal subgroups of $\tilde S$ whose orders are divisible by such a prime. We deduce, with the aid of \cite[Table 3.5.E]{kleidmanliebeck} and \cite[Table 2.5]{bhrd}, that if $\mathcal{M}_{\tilde S}(\tilde s)$ contains a subgroup $\tilde M$ that is not reducible or an extension field subgroup, then $m+3$ is the unique ppd of $q^{m+2}-1$, and $\tilde M$ is an extension of a (possibly trivial) $2$-group by an alternating or symmetric group of degree at most $2m+2$. As $|\tilde s|$ is odd, $\sym_{2m+2}$ must then contain an element of order $|\tilde s|$. However, Lemma \ref{lem:symelt} shows that this is not the case. Therefore, each subgroup in $\mathcal{M}_{\tilde S}(\tilde s)$ is reducible or an extension field subgroup.

It is clear from Lemma \ref{lem:orthplussubspaces} that the stabiliser
of $U$ in $\tilde S$ is the unique reducible subgroup in
$\mathcal{M}_{\tilde S}(\tilde s)$. Thus it remains to show that
$\mathcal{M}_{\tilde S}(\tilde s)$ contains no extension field
subgroup $\tilde Y$ that is not of type $\mathrm{GU}_m(q)$. By
\cite{gpps}, any such $\tilde Y$ whose order is divisible by a
ppd of $q^{m+2}-1$ is defined over $q^2$, and so
$\tilde Y$ is an extension of $\Omega^+_m(q^2)$ by a group of order
$4$ \cite[Table 2.6]{bhrd}. The $p'$-part of the order of $\Omega^+_m(q^2)$ divides $(q^m-1) \prod_{i = 1}^{m/2-1}(q^{4i} - 1)$, and for each positive integer $f \in \{m+3, \ldots, 2m-4\}$,
no ppd of $q^{m+2} - 1$ divides $q^{f} - 1$.
We conclude that $|\tilde Y|$ is not divisible by any ppd of $q^{m+2}-1$, and therefore $\tilde Y \notin \mathcal{M}_{\tilde S}(\tilde s)$.
\end{proof}

In order to prove a similar proposition when $m = 4$, we consider maximal tori of $\tilde S$ (or $\mathrm{SO}^+_{2m}(q)$ if $q$ is odd) that contain $\tilde s$, via a brief discussion of related algebraic groups. We in fact allow $m$ to be any even integer at least $4$, as the results here will also be useful when considering extension field subgroups in $\mathcal{M}_{\tilde S}(\tilde s)$ in each case.

Let $\underline S$ be the algebraic group $\mathrm{SO}_{2m}(\overline{\mathbb{F}_q})$ for some $m \ge 4$. There exists a Frobenius endomorphism $\sigma$ of $\underline S$ such that the subgroup $\hat S:=\underline S_\sigma$ of fixed points of $\underline S$ under $\sigma$ is equal to $\mathrm{SO}^+_{2m}(q)$ if $q$ is odd, or $\Omega^+_{2m}(q)$ if $q$ is even (see \cite[pp.~193--194]{malletesterman}, where $\Omega^+_{2m}(q)$ is written as $\mathrm{SO}^+_{2m}(q)$ when $q$ is even). By definition, 
each maximal torus $\hat T$ of $\hat S$ is equal to $
  \underline T \cap \hat S$ 
for a corresponding  maximal torus $\underline T$ of $\underline S$. An element $\tilde r$ of $\hat T$ is called \emph{regular} if the dimension of $C_{ \underline S}(\tilde r)$ is no larger than the dimension of $C_{ \underline S}(\tilde x)$ for any $\tilde x \in \underline S$, or equivalently, if $\underline T$ is equal to $C_{\underline S}(\tilde r)^\circ$, the connected component of $C_{ \underline S}(\tilde r)$ containing the identity \cite[Defn.~14.8, Cor.~14.10]{malletesterman}. We will continue to use this notation in the following two results and their proofs.

\begin{lemma}
\label{lem:sotorus}
Let $\tilde r$ be a regular element of a maximal torus $\hat T$ of $\hat S$. If $|\tilde r|$ is odd, then $C_{\hat S}(\tilde r) = \hat T$.
\end{lemma}

\begin{proof}
As $\tilde r$ is regular, $\underline T = C_{ \underline S}(\tilde r)^\circ$. In addition, it follows from \cite[pp.~70--72, Prop.~14.20]{malletesterman} that the order of $C_{ \underline S}(\tilde r)/C_{ \underline S}(\tilde r)^\circ$ divides $(2,q-1)$, and its exponent divides $|\tilde r|$. Hence if $|\tilde r|$ is odd, then $C_{ \underline S}(\tilde r) = \underline T$, and so $C_{\hat S}(\tilde r) =  \underline T \cap \hat S = \hat T$.
\end{proof}

\begin{lemma}
\label{lem:omegapluscentraliser}
Let $\tilde S$ and $\tilde s$ be as in Assumption \ref{assump:selt}. Then the centraliser $C_{\hat S}(\tilde s) = C_{\mathrm{GO}^+_{2m}(q)}(\tilde s)$ is a maximal torus of $\hat S$ of shape $(q^a+1) \times (q^b+1)$, where $a + b = m$, with $a:= 2$ if $m = 4$ and $a:= (m-2)/2$ otherwise.
\end{lemma}

\begin{proof}
  As $\tilde s$ is semisimple, it lies in a maximal torus $\hat T$ of $\hat S$. By considering the 
    maximal tori of $\hat S$ (see \cite[\S4--5]{buturlakin}, \cite[\S9]{veldkamp} and \cite[Lemma 2]{seitzstructure}), we deduce that if $\hat T$ does not have shape $(q^a+1) \times (q^b+1)$, then $m = 4$ and $\hat T$ is a cyclic group of order $q^4-1$ that stabilises a totally singular four-dimensional subspace of $\mathbb{F}_q^8$. Lemma~\ref{lem:orthplussubspaces} shows that this second case cannot occur.

Recall that when $m = 4$, the elements $\tilde r_1$ and $\tilde r_2$ from Assumption \ref{assump:selt} have distinct eigenvalues on $\mathbb{F}_q^8$. Additionally, even though $q^{2a}-1$ has no ppd when $m = 8$ and $q = 2$, the element $\tilde r_1$ in this case has order $9$, and no element of this order lies in a proper subfield of $\mathbb{F}_{q^{2a}} = \FF_{2^6}$. We therefore deduce from \cite[\S9]{veldkamp} in each case that $\tilde s^2$ is a regular element of $\hat T$ (as is $\tilde s$). Since $C_{\hat S}(\tilde s) \le C_{\hat S}(\tilde s^2)$ and $|\tilde s^2|$ is odd, we obtain $C_{\hat S}(\tilde s) = \hat T$ from Lemma~\ref{lem:sotorus}.

Finally, it follows from Lemma \ref{lem:orthplussubspaces} if $4 \mid m$, or from the definition of $\tilde s$ in Assumption \ref{assump:selt} if $4 \nmid m$, that $\langle \tilde s \rangle$ stabilises no one-dimensional subspace of $\mathbb{F}_q^n$. Hence $\tilde s$ has no linear elementary divisor, and so \cite[p.~85--86]{defranceschi} yields $C_{\hat S}(\tilde s) = C_{\mathrm{GO}^+_{2m}(q)}(\tilde s)$.
\end{proof}

We can now prove a version of Proposition~\ref{prop:omegamultsubgps} with $m = 4$.

\begin{prop}
\label{prop:omega8subgps}
Let $\tilde S$ and $U$ be as in Assumption \ref{assump:selt}, with $m = 4$ and $q \ge 7$. Then the element $\tilde s$ from Assumption \ref{assump:selt} can be chosen so that the stabiliser of the decomposition $\mathbb{F}_q^{2m} = U \perp U^\perp$ in $\tilde S$ is the unique subgroup in $\mathcal{M}_{\tilde S}(\tilde s)$ that is not an extension field subgroup of type $\mathrm{GO}^+_m(q^2)$.
\end{prop}

\begin{proof}
We shall initially consider an arbitrary choice for $\tilde s$, and work through the families of maximal subgroups of $\tilde S$, given in \cite[Table 8.50]{bhrd} (see also \cite[Table I]{kleidmano8}).
Let $\tilde L$ be the stabiliser in $\tilde S$ of the decomposition of $V:= \mathbb{F}_q^{8}$ as $U \perp U^\perp$. Then $\tilde L$ contains the stabiliser in $\tilde S$ of $U$. Lemma \ref{lem:orthplussubspaces} shows that $U$ and $U^\perp$ are the only proper non-zero subspaces of $V$ stabilised by $\langle \tilde s \rangle$. Therefore, $\mathcal{M}_{\tilde S}(\tilde s)$ contains no reducible subgroups. Moreover, if $\langle \tilde s \rangle$ stabilises a decomposition $V = W_1 \oplus W_2$, where $\dim(W_1) = \dim(W_2) = 4$, then $\langle \tilde s^2 \rangle = \langle \tilde s \rangle$ stabilises each of $W_1$ and $W_2$, and hence $\{W_1,W_2\} = \{U,U^\perp\}$. Thus $\tilde L$ is the unique stabiliser of such an imprimitive decomposition that contains $\tilde s$.

Next, let $\tilde E$ be an extension field subgroup of $\tilde S$ of type $\mathrm{GU}_4(q)$, and recall that $\hat S =\mathrm{SO}^+_{2m}(q)$ if $q$ is odd, and $\hat S =\tilde S$ if $q$ is even. We see from \cite[Construction 2.5.14, Lemma 5.3.6]{burnessgiudici} that $\hat E:=N_{\hat S}(\tilde E)$ has shape $\mathrm{GU}_4(q).2$. Thus by \cite[Cor.~2]{buturlakin} (see also \cite[Lemma 2]{seitzstructure}), each (semisimple) element of $\hat E$ of order $k:=|\tilde s|$ lies in a cyclic maximal torus of $\mathrm{GU}_4(q)$ of order $q^4-1$. However, Lemma \ref{lem:omegapluscentraliser} shows that $C_{\hat S}(\tilde s)$ is not cyclic. We therefore deduce that $\tilde s \notin \tilde E$.

For a similar argument, let $X$ be a maximal subgroup of $S$ that is isomorphic to $\Omega_7(q)$ if $q$ is odd, or to $\mathrm{Sp}_6(q)$ if $q$ is even. Then each element of $X$ of order 
$|s| = k$
lies in a maximal torus of shape $k \times (q+1)$ or $k \times (q-1)$ \cite[Thms.~3--4]{buturlakin}. Additionally, $C_S(s)$ is abelian of order $k^2$ \cite[Thms.~5 \& 7]{buturlakin}. As $k^2$ is divisible by neither $q+1$ nor $q-1$ (since $q \ne 2$), we see that $s \notin X$ and $\tilde s \notin \tilde X$.

Suppose now that $q$ is odd, and let $H$ be a maximal subgroup of $S$ 
of shape $(\psl_2(q) \times \psp_4(q)).2$. Since $|\psl_2(q)|$ and $|s|$ are coprime, any element of $H$ of order $k$ has a centraliser that contains the non-abelian group $\psl_2(q)$. Thus $\tilde s \notin \tilde H$.

It now follows from \cite[Table I]{kleidmano8} and \cite[Table 8.50]{bhrd} that if $\tilde Y$ is any other maximal subgroup of $\tilde S$ whose order is divisible by 
$k$, such that $\tilde Y$ is not an extension field subgroup of type $\mathrm{GO}^+_m(q^2)$, then $q$ is a square and $Y \cong \Omega^-_8(\sqrt{q})$. 

 For a subspace $W$ of $V$ and an element $g \in \tilde S$, let $g_W$ denote the set of $\mathbb{F}_{q^8}$-eigenvalues of $g$ on $W$. Recall from Assumption~\ref{assump:selt} that $\tilde s_U = \mathcal{A}:= \{\alpha, \alpha^q, \alpha^{q^2},\alpha^{q^3}\}$ and $\tilde s_{U^\perp} = \mathcal{B}:=\{\beta, \beta^q, \beta^{q^2}, \beta^{q^3}\}$, where $\alpha$ and $\beta$ are elements of $\mathbb{F}_{q^4}$ of order $k$ and $\mathcal{A} \cap \mathcal{B} = \varnothing$. Let $\mathcal{Y}$ be the set of maximal subgroups $\tilde Y$ of $\tilde S$ such that $Y \cong \Omega^-_8(\sqrt{q})$, and for $\tilde Y \in \mathcal{Y}$, let $\mathcal{J}({\tilde Y})$ be the set of elements $y \in \tilde Y$ of order $k$, such that the proper non-zero subspaces of $V$ stabilised by $y$ are precisely $U$ and $U^\perp$; the set $y_V$ contains $\mathcal{A}$; and $y_V \setminus \mathcal{A} = \{\gamma, \gamma^q, \gamma^{q^2}, \gamma^{q^3}\}$ for some $\gamma \in \mathbb{F}_{q^4}$ of order $k$. Additionally, let \[\mathcal{E}:= \left \{ y_V \setminus \mathcal{A} \mid y \in \mathcal{J}({\tilde Y}), \tilde Y \in \mathcal{Y} \right \}.\]
We will determine an upper bound for $|\mathcal{E}|$. Using this bound, we will show that the element $\tilde r_2$ in Assumption \ref{assump:selt} (equivalently, the set $\mathcal{B}$) can be chosen so that, for all $\tilde Y \in \mathcal{Y}$, the set $\mathcal{J}({\tilde Y})$ does not contains $\tilde s$, and so $\tilde s \notin \tilde Y$.

 Fix $\tilde Y \in \mathcal{Y}$ and $y \in \mathcal{J}({\tilde Y})$. Notice that $\{\mathcal{A}, y_V \setminus \mathcal{A}\} = \{y_U, y_{U^\perp}\}$, and that a given $c \in \langle y \rangle \cap \mathcal{J}({\tilde Y})$ satisfies $c_U = y_U$ if and only if $c_{U^\perp} = y_{U^\perp}$ (additionally, at least one of $c_U$ and $c_{U^\perp}$ is equal to $\mathcal{A}$).
 We also deduce from \cite[\S4--5]{buturlakin} (and the fact that $Y \cong \mathrm{PSO}^-_8(\sqrt{q})$ if $q$ is odd) that any two cyclic subgroups of $\tilde Y$ of order $k$ are conjugate in $\tilde S$. Thus the fixed $\tilde Y \in \mathcal{Y}$ contributes two (not necessarily distinct) sets to $\mathcal{E}$: the set $y_{U^\perp}$ for elements $y \in \mathcal{J}({\tilde Y})$ such that $y_U = \mathcal{A}$, and the set $y_{U}$ for elements $y \in \mathcal{J}({\tilde Y})$ such that $y_{U^\perp} = \mathcal{A}$.

We observe from Rows 64--69 and Columns V and XII of \cite[Table I]{kleidmano8} that the subgroups in $\mathcal{Y}$ form exactly two $N_{\gl_8(q)}(\tilde S)$-conjugacy classes.
Thus there exist $\tilde Y_1, \tilde Y_2 \in \mathcal{Y}$ such that \[\mathcal{E}= \left \{ y_V \setminus \mathcal{A} \mid y \in \mathcal{J}({\tilde Y_1}) \cup \mathcal{J}({\tilde Y_2}) \right \}.\]
By the previous paragraph, each of $\tilde Y_1$ and $\tilde Y_2$ contributes at most two sets to $\mathcal{E}$, and thus $|\mathcal{E}| \le 4$.

%

To ensure that $\tilde s$ does not lie in $\mathcal{J}({\tilde Y})$ for any $\tilde Y \in \mathcal{Y}$ (and hence that $\tilde s \notin \tilde Y$), it suffices to choose $\tilde r_2$ in Assumption \ref{assump:selt} so that the corresponding set $\mathcal{B}$ intersects trivially with each of the sets in $\mathcal{E}$, and with $\mathcal{A}$ (as required to satisfy Assumption \ref{assump:selt}). Since each of these sets of eigenvalues has size four, this is possible as long as $\varphi(k)/4 \ge 6$, where $\varphi$ is Euler's totient function.  It is well known (see, e.g., \cite[Prop.~2]{nicolas}) that $\varphi(i) \ge \sqrt{i/2}$ for each $i$. Hence an appropriate $\tilde r_2$ exists if $k \ge 1152$. If instead $k < 1152$, then the square $q \ge 9$ is at most $25$, and again $\varphi(k)/4 \ge 6$.
\end{proof}

We observe from \cite[Prop.~5.14]{bgk} that, when $4 \nmid m$, the subgroup $\tilde L$ described in Theorem \ref{thm:orthogplusmaxes} is the unique subgroup in $\mathcal{M}_{\tilde S}(\tilde s)$ that is not an extension field subgroup. Thus to prove Theorem \ref{thm:orthogplusmaxes}, it remains to show, for all even $m \ge 4$, that $\mathcal{M}_{\tilde S}(\tilde s)$ contains exactly two extension field subgroups of the specified type, and that they are not conjugate in $\tilde S$.

\begin{prop}
\label{prop:orthplusextfield}
Let $\tilde S$ and $\tilde s$ be as in Assumption \ref{assump:selt}. Then $\mathcal{M}_{\tilde S}(\tilde s)$ contains exactly two extension field subgroups of type $J$, where $J \cong \mathrm{GU}_m(q)$ if $4 < m \equiv 0 \pmod 4$, and $J \cong \mathrm{GO}^+_m(q^2)$ otherwise. Additionally, the two subgroups are not conjugate in $\tilde S$.
\end{prop}

\begin{proof}
    Let $a$, $b$, and $\hat S$ 
    be as in Lemma~\ref{lem:omegapluscentraliser}, so that $C_{\hat S}(\tilde s) = C_{\mathrm{GO}^+_{2m}(q)}(\tilde s)$ is a maximal torus $\hat T$ of $\hat S$ of shape $(q^a+1) \times (q^b+1)$.
  Additionally, let $\tilde K$ be a 
    subgroup of $\tilde S$ of type $J$. Then \cite[Constructions 2.5.13--2.5.14, Lemmas 5.3.4 \& Lemma 5.3.6]{burnessgiudici} and \cite[Tables 3.5.E \& 3.5.G]{kleidmanliebeck} imply that \[\hat K:=N_{\hat S}(\tilde K) \cong J.\langle \phi \rangle,\] where $\phi$ is the involutory field automorphism of the 
    subgroup $\SU_m(q)$ or $\Omega^+_m(q^2)$ of $J$. As all maximal tori of $\hat S$ isomorphic to $\hat T$ are conjugate in $\hat S$ \cite[p.~394]{veldkamp}, and as $J$ contains such a maximal torus (see \cite[Cor.~2]{buturlakin}, \cite[\S9]{veldkamp} and \cite[Lemma 2]{seitzstructure}), some $\hat S$-conjugate of $\hat K$ contains $\hat T$, and hence contains $\tilde s$. Now, by \cite[Tables 8.50 \& Table 8.82]{bhrd} and \cite[Tables 3.5.E \& 3.5.G]{kleidmanliebeck}, $\tilde S$ has exactly two conjugacy classes of maximal subgroups that are extension field subgroups of type $J$, and these extend to two conjugacy classes of maximal subgroups of $\hat S$. We have shown that each of these two conjugacy classes has at least one subgroup containing $\tilde s$.

    To complete the proof, we will show that $\mathcal{M}_{\tilde S}(\tilde s)$ contains exactly two extension field subgroups of type $J$. If $4 \nmid m$, then this is the case by \cite[Prop.~5.14]{bgk}. If instead $4 \mid m$, then $|s|$ is odd. Thus \cite[p.~34, pp.~38--39]{wall} (see also \cite[Prop.~2.11]{bgk}) implies that any two similar elements of $\mathrm{GO}^+_{2m}(q)$ are conjugate, as are any two similar elements of $J$. Additionally, conjugating $\tilde K$ by an element of $\hat S$ if necessary, we may assume that $\hat T \le \hat K$. Therefore, arguing as in the proof of \cite[Lemma 5.9]{bgk}, we deduce that
    the elements of $\tilde s^{\mathrm{GO}^+_{2m}(q)} \cap \hat K$
    form
  exactly two $\hat K$-conjugacy classes.

Now, any two extension field subgroups of $\tilde S$ of type $J$ extend to conjugate subgroups of $\mathrm{GO}^+_{2m}(q)$ \cite[Tables 3.5.E \& Table 3.5.G]{kleidmanliebeck}, and so $\hat K$ is self-normalising in $\mathrm{GO}^+_{2m}( q)$. Moreover, by considering the maximal tori of $J$ (again, see \cite[Cor.~2]{buturlakin}, \cite[\S9]{veldkamp} and \cite[Lemma 2]{seitzstructure}), we deduce that $C_{\mathrm{GO}^+_{2m}(q)}(\tilde f) \le \hat K$ for each $\tilde f \in \tilde s^{\mathrm{GO}^+_{2m}(q)} \cap \hat K$. Therefore, Lemma \ref{lem:conjclassmax}\ref{conjclassmax2} implies that $\tilde s$ lies in exactly two $\mathrm{GO}^+_{2m}(q)$-conjugates of $\hat K$, and hence in exactly two $\mathrm{GO}^+_{2m}(q)$-conjugates of $\tilde K$.
\end{proof}

Theorem \ref{thm:orthogplusmaxes} now follows from \cite[Prop.~5.14]{bgk} (with $4 \nmid m$) and Propositions \ref{prop:omegamultsubgps}, \ref{prop:omega8subgps} and \ref{prop:orthplusextfield}. We are also now able to prove the final case of Theorem \ref{thm:maxsubnoconj}, and hence complete the proof of Theorem \ref{thm:noncomint}.


\begin{prop}
Theorem \ref{thm:maxsubnoconj} holds when $S = \oo^+_{2m}(q)$, with $m \ge 4$ even.
\end{prop}

\begin{proof}
  By \cite[Lemma 10]{galt}, $s$ is $S$-conjugate to its inverse, and so $N_S(\langle s \rangle) > C_S(s)$. Additionally, Theorem \ref{thm:orthogplusmaxes} shows that $\mathcal{M}_S(s)$ consists of three subgroups, $K_1$, $K_2$ and $L$, no two of which are conjugate in $S$. If $m = 4$, then the intersection of any two of these subgroups is equal to the intersection of all three \cite[Lemma 5.15(e)]{bgk}, and hence Lemma \ref{lem:conjmaxint} shows that $\mathcal{M}'_G(s)$ contains at most one novelty maximal. If instead $m > 4$, then $L$ is isomorphic to neither $K_1$ nor $K_2$ (see Table 3.5.E and the corresponding results in Chapter 4 of \cite{kleidmanliebeck}). 
 In each case, Lemma \ref{lem:ordinarynonconj} yields the result.
\end{proof}


\section{Finite groups satisfying the independence property are supersoluble}\label{riduco}

The aim of this section is to prove that each finite group $G$ satisfying the independence property is supersoluble. We will reduce the proof of this statement to the case where $G$ is almost simple, and finally, in the almost simple case, we will reach our conclusion by applying Theorem \ref{thm:noncomint}. Two elements $x$ and $y$ of a group $G$ are \emph{dependent} if no minimal generating set for $G$ contains $\{x, y\}$. 

We shall assume throughout this section that $G$ is a finite group satisfying the independence property.
Our first result reduces the study of
such groups
to the 
case where their Frattini subgroup is trivial.

\begin{prop}\label{frattini}
  The Frattini subgroup $\frat(G) \neq 1$ if and only if $G$ is either a cyclic $p$-group or the quaternion  group  of order 8.
\end{prop}
\begin{proof}
  Suppose that 
    $\frat(G) \neq 1$, and let $x$ be an element of $\frat(G)$ of prime order. Since $x$ is a non-generator of $G$, $x$ and $y$ are dependent for each $y \in G \setminus \{x\}$.
    In particular, since $|x|$ is prime, $\langle x \rangle \subseteq \langle y \rangle$ for every $y \in G\setminus \{1\}$. Thus $\langle x\rangle$ is the unique minimal subgroup of $G$ and therefore either
$G$ is a cyclic $p$-group or
$G=\langle a, b \mid a^{2^{n-1}}=1, 
b^2= a^{2^{n-2}}, a^b=a^{-1}\rangle$  (with $n \ge 3$) is a generalized
quaternion group (see \cite[Thm.~9.7.3]{wrscott}). In the second case, $a^2\in \frat(G)$ and therefore $b$ and $ba^2$ are dependent. Since $G$ satisfies the independence property, 
$\langle b, ba^2\rangle$ is cyclic, and hence $[b,a^2]=1$.
This implies that $n=3$, as required. The converse is clear. 
\end{proof}

We now present a sequence of results characterising the minimal normal subgroups of our group $G$.

\begin{lemma}\label{ciclici}
  Suppose that 
  $\frat(G) = 1$, and that $G$ has an abelian minimal normal subgroup $N$. Then $N$ is cyclic of prime order. 
 \end{lemma}
\begin{proof}
  Let $n_1$ and $n_2$ be distinct non-trivial elements of $N$. We claim that they are dependent. Indeed, assume for a  contradiction that $X:=\{n_1,n_2,g_1,\dots,g_t\}$ is a minimal generating set for $G$ and let $H:=\langle g_1, \dots,  g_t\rangle.$ Since \[G= \langle n_1,n_2,g_1,\dots,g_t \rangle\leq HN,\]
	$G=HN$ and $N$ is $H$-irreducible. But then $G= \langle n_1,g_1,\dots,g_t \rangle$, contradicting the minimality of $X$.

        Since $n_1$ and $n_2$ are dependent,
	one of 
        them is a power of the other. Since $N$ is elementary abelian, it follows that $\langle n_1 \rangle=
	\langle n_2 \rangle.$ Therefore,
 $N\cong C_p.$
\end{proof}

\begin{lemma}\label{simpleex}
  Suppose that 
    $\frat(G) = 1$, and that $G$ has  a non-abelian minimal normal subgroup $N$. If $\oo^+_8(2)$ is not a composition factor of $N,$ then $N$ is simple.
	\end{lemma}
\begin{proof}
Here, $N=S_1\times \dots \times S_t$ with $S_i \cong S$ a non-abelian finite simple group. Assume for a contradiction that $S\neq\oo^+_8(2)$ and $t\geq 2.$ By \cite[Cor.~7.2]{guma}, there exist $x$ and $y$ in $S$ such that $S=\langle x^\alpha, y^\beta\rangle$
for every $\alpha, \beta \in \aut(S)$. Notice that in particular $y \neq x^{\gamma}$ for all $\gamma \in \aut(S)$.

 Let $n:=(x,y,1,\dots,1)\in N$ and choose $g_1,\dots,g_d \in G$
 such that $G=\langle g_1,\dots, g_d\rangle N.$ Set $H:=\langle n, g_1,\dots,g_d\rangle$ and denote by $\pi_i: N\to S_i$ the projection to the $i$-th factor of $N.$ The factors $S_1,\dots,S_t$ are permuted transitively by $H$, so in particular there exists $h\in H$ such that
 $S_2^h=S_1$.  Thus 
$\pi_1(n^h)=y^\beta$ for some $\beta \in \aut(S).$

 It follows that 
$S_1=\langle x, y^\beta \rangle \leq \langle \pi_1(n), \pi_1(n^h) \rangle \leq \pi_1(H\cap N)$, and since $H$ is transitive on $\{S_1,\dots,S_t\}$, the image $\pi_i(H\cap N)=S_i$ for all $i \in \{1, \ldots, t\}$.
Therefore,  there exists a partition  of $\{1,\dots,t\}$  into $u$ blocks of size $ v:= t/u$, say  $J_i:=\{j_{i1}, \ldots, j_{iv}\}$ for $1 \leq i \leq u$, 
 corresponding elements $\alpha_{ik} \in \aut(S)$,  and diagonal subgroups $\Delta_i:=\{(s^{\alpha_{i1}},\dots,s^{\alpha_{iv}})\mid s \in S\}$ of $S_{j_{i1}}\times \dots \times S_{j_{iv}}$, such that $H\cap N=\Delta_1\times \dots \times \Delta_u$.
We may assume
that $1=j_{11} \in J_1.$

Since $n\in  H\cap N,$ it follows  that 
$J_1\subseteq \{1,2\}.$ If $J_1=\{1,2\},$ then \[\Delta_1=\{(s,s^\gamma)\mid s \in S\}\leq S_1\times S_2,\] with $\gamma=\alpha_{11}^{-1}\alpha_{12}$, and $n\in \Delta_1.$ This implies that $y=x^\gamma$, 
  a contradiction. Hence $J_1=\{1\}$,  so $|J_i| = 1$ for all
  $i$ and 
$H\cap N = S_1 \times \cdots \times S_t  =N$. We conclude that $G = H = \langle n, g_1, \ldots, g_d \rangle$.

We have shown that $\langle n, g_1, \ldots, g_d \rangle = G$ for all $g_1, \ldots, g_d \in G$ such that $\langle g_1, \ldots, g_d \rangle N = G$. Thus no minimal generating set for $G$ contains both $n$ and $m:=(y,1,\dots,1) \in N$.
  Consequently,
%
one of $n$ and $m$ is a power of the other. This implies in particular that $\langle x, y \rangle$ is cyclic, contradicting  $\langle x, y \rangle=S.$
\end{proof}

The restriction in the previous lemma concerning 
 $\oo^+_8(2)$ 
 can be removed by a different argument. For this purpose we need the following lemma.
In the next few results, we shall denote elements $h \in X \wr \sym_t$ by $(\rho_1(h),
\ldots, \rho_t(h))\sigma(h)$, where $\rho_i$ is the projection from
the base group to the $i$-th copy of $X$ and $\sigma(h) \in \sym_t$.

 \begin{lemma}\label{strano}
 Let $X$ be a finite group, let $a$ be an element of $X$,
  let $t \ge 2$, and let $\alpha:=(a,a^2,1,\dots,1)$
and  $\beta:=(a,1,\dots,1)$ be elements of $X \wr \sym_t$. Let $Y$ be
a subgroup of $X \wr \sym_t$ and set $R:=\langle \alpha, Y\rangle$
and $K:=\langle \alpha, \beta,Y\rangle.$ For any given $k \in K$ and $i\in \{1,\dots,t\}$, there exists $r\in R$ such that $\rho_i(r)=\rho_i(k)$ and $\sigma(r)=\sigma(k).$
 \end{lemma}
\begin{proof}
Let $A:=\langle \alpha \rangle$ and $B:=\langle \beta \rangle.$
Any element $k\in K$ can be written in the form $k=z_1\cdots z_\ell$ with
$z_i \in A \cup B \cup Y.$ We prove the statement by induction on $\ell.$
First assume $\ell=1.$ If $k=z_1 \in Y \cup A,$ we set $r=k.$ If
$k=z_1=\beta^m \in B$, we set $r=\alpha^m$ if $i=1,$ and $r=1$ otherwise.

Now assume $\ell>1$ and set $k^*=z_2\cdots z_\ell.$ We have
$\sigma(k)=\sigma(z_1)\sigma(k^*)$ and
$\rho_i(k)=\rho_i(z_1)\rho_{i\sigma(z_1)}(k^*).$ By induction, 
  for each $i$  there exist $r_1, r_2\in R$ such that $\sigma(r_1)=\sigma(z_1),$  $\sigma(r_2)=\sigma(k^*),$ $\rho_i(r_1)=\rho_i(z_1)$, and  $\rho_{i\sigma(z_1)}(r_2)=\rho_{i\sigma(z_1)}(k^*).$ The element $r=r_1r_2$ satisfies the required properties.
\end{proof}

\begin{lemma}\label{semplice}
  Suppose that 
  $\frat(G) = 1$, and that $G$ has a non-abelian minimal normal
  subgroup $N$. Then $N$ is simple.
\end{lemma}

\begin{proof}
 We have $N=S_1\times \dots \times S_t$, with $S_i \cong S$ a non-abelian finite
 simple group. 
 By Lemma \ref{simpleex}, we may assume that
 $S = \oo^+_8(2)$. In
 particular, $S$ contains an element $a$ of order four. 

 Assume for a contradiction that $t\geq 2$, and let
 $\alpha:=(a,a^2,1,\dots,1)$ and $\beta:=(a,1,\dots,1).$ Suppose that $G=\langle \alpha, \beta, g_1,\dots,g_d\rangle$ for some $g_1, \ldots, g_d \in G$, and let $Y:=\langle 
g_1,\dots,g_d\rangle$ and $R:=\langle 
 \alpha,g_1,\dots,g_d\rangle$. We apply Lemma~\ref{strano} to 
 $RC_G(N)/C_G(N)$ and  $G/C_G(N)\leq \aut(N)=\aut(S) \wr \sym_t$, with the element $(s, s, \ldots, s)C_G(N) \in G/C_G(N)$,  for an arbitrary $s \in S$, with $i = 1$. We deduce that there  exist $\gamma_2,\dots, \gamma_t \in \aut(S)$ 
 such that $(s,\gamma_2,\dots,\gamma_t)C_G(N) \in RC_G(N)$.

Let $\pi_1$ be the projection from $N$ to $S_1$. Since $S$ is non-abelian simple and $\alpha \in R \cap N \unlhd
  R$, we deduce that $\pi_1(R\cap
  N)=S$. Since $RN=G$, the action of $R$ on 
    $S_1, \dots,
S_t$ is transitive and consequently, arguing as in the proof of
  Lemma~\ref{simpleex}, we see that
  $ R\cap N=\Delta_1\times \dots\times \Delta_u$, where
  $\Delta_i:=\{(s^{\alpha_{i1}},\dots,s^{\alpha_{iv}})\mid s \in S\}$
  is a diagonal subgroup of $S_{j_{i_1}}\times \dots \times
  S_{j_{i_v}}$ and each $\alpha_{ik} \in \aut(S)$.
On the other hand, $\alpha \in R\cap N$, so the order of $a$ 
implies that  $v=1$, hence $ R\cap
N=N$ and $R= RN=G.$ This implies that $\alpha$ and $\beta$ are
dependent, a contradiction since neither of $\alpha$ and $\beta$ is a power of the other.
\end{proof}

\begin{lemma}\label{due}
  Let $N_1$ and $N_2$ be two distinct minimal normal subgroups of a finite group $X$, and let $a\in N_1$ and $b\in N_2.$ If $ab$ and $b$ are independent in $X$, then 
  there exists an isomorphism $\phi :N_1\to N_2$ such that $\phi(a)=b$,  and such that $\phi$ is an $X$-isomorphism if $N_1$ is abelian.
 \end{lemma}

\begin{proof}
	Assume that $ab$ and $b$ are independent, and let $\{ab, b, x_1,\dots, x_d\}$ be a minimal generating set for $X.$ Additionally, let
	$H:=\langle ab,  x_1,\dots, x_d\rangle$, so $H \neq X$. We have $b=a^{-1}(ab)\in N_2\cap HN_1$ and therefore $X=\langle H, b \rangle=HN_1=HN_2.$ Moreover
	$H\cap N_1$ is normalized by $H$ and centralized by $N_2,$ so
	$H\cap N_1$ is normal in $X=HN_2$ and therefore $H\cap N_1=1.$ Similarly $H\cap N_2=1,$ so $H$ is a common complement of $N_1$ and $N_2$ in $X.$ In particular, for any $n$ in $N_1$ there exists a unique $n^*$ in $N_2$ such that $nn^*\in H$, and since $[N_1,N_2] = 1$, the map $\phi: N_1\to N_2$ sending $n$ to $n^*$ is an isomorphism.  Now $ab\in H,$ so $\phi(a)=b.$

	Suppose now that $N_1$ is abelian, and let $n \in N_1$ and $x \in X$. Then $x = n'h$ for some $n' \in N_1$ and $h \in H$. As $[N_1,N_1] = 1 = [N_1,N_2]$, we see that $(n n^*)^x = (n n^*)^h \in H$. Thus $\phi(n^x) = (n^*)^x = \phi(n)^x$, and so $\phi$ is an $X$-isomorphism.
 \end{proof}


\begin{lemma}\label{unico}
  The group $G$ contains at most one non-abelian minimal normal subgroup.
\end{lemma}
\begin{proof}  If $\frat(G) \neq 1$ then the result is immediate from Proposition \ref{frattini}, so assume $\frat(G)=1.$
  Assume further, for a  contradiction, that $N_1$ and $N_2$ are distinct non-abelian minimal normal subgroups of $G$. By Lemma~\ref{semplice},  $N_1$ and $N_2$ are simple. Therefore, there exist $a\in N_1$ and $b\in N_2$ such that  $|a|=|b|\neq 1.$

  Since neither of $ab$ and $b$ is a power of the other, they are independent and therefore, by Lemma \ref{due}, there exist a non-abelian simple group $S$ with $N_1 \cong N_2 \cong S$, and an element $\phi \in \aut(S)$ such that $b=\phi(a).$ So $S$ is a non-abelian simple group such that all elements of the same order are conjugate in its automorphism group. By \cite[Thm.~3.1]{jpz}, $S\in \{\psl_2(5), \psl_2(7), \psl_2(8), \psl_2(9), \psl_3(4)\}$.

  Assume that $S$ contains an element $a$ of order $p^2$, for some prime $p,$ and consider $(a,a^p)$ and $(1,a^p)$ in $N_1\times N_2.$ Neither of these elements is a power of the other, so they are independent, contradicting Lemma~\ref{due}. So $S$ contains no elements of order $p^2$, and therefore $S \cong \psl_2(5) \cong \alt_5$. 

  Here, we consider an element $u$ of order $5$ in $S$ and an involution  $t$ in $\mathrm{N}_S(\langle u \rangle)$. Then $M:=\langle t, u\rangle$ is isomorphic to the dihedral group of order 10,  and is the unique maximal subgroup of $S$ containing $u.$ Let $x:=(t,u), y:=(1,t) \in N_1\times N_2\cong S^2.$  Neither of $x$ and $y$ is a power of the other, so they are independent.

  Let $\{x, y, g_1,\dots,g_d\}$ be a minimal generating set for $G$, and define $H:=\langle x, g_1,\dots, g_d\rangle.$ We have $HN_2=G$,  since $y \in N_2.$
  Let $X:=HN_1\cap N_1N_2.$ Now $(t, u) \in H$ and $(t^{-1}, 1) \in N_1$, so $(1, u) \in HN_1 \cap N_1 N_2 = X$. Furthermore, $\langle (1, t), (1, u) \rangle N_1$ is the unique maximal subgroup of $N_1N_2$ that contains $\langle (1, u), N_1 \rangle$, so if $X\neq N_1N_2,$ then $X \leq \langle (1,t), (1,u)\rangle N_1$.
  But then $y$ normalizes $X$ and consequently $X \unlhd G,$ contradicting the minimality of $N_2$.
  
  Thus $X = N_1N_2 \leq HN_1$ 
  and using  Dedekind's modular law and $HN_2 = G$ we see that \[(H\cap N_1N_2)N_1 =  N_1N_2 =(H\cap N_1N_2)N_2.\] In particular $H \cap N_1 N_2$ is a subdirect product of $N_1 N_2 \cong
 \alt_5^2$. If $N_1N_2\leq H,$ then $G=N_1N_2\langle  g_1,\dots, g_d\rangle\leq H,$ a contradiction. Hence $H\cap N_1N_2 = \{(s,s^\phi)\mid s \in S\}$ for some $\phi \in \aut(S)$, but this contradicts $(t,u)\in H\cap N_1N_2.$
\end{proof}

In the following, we write $F(G)$ for the Fitting subgroup of $G$ (the largest normal nilpotent subgroup of $G$), $E(G)$ for the subgroup of $G$ generated by all quasisimple subnormal subgroups of $G$,  and $F^*(G)$ for the generalized Fitting subgroup of $G$ (the subgroup generated by $F(G)$ and $E(G)$).

\begin{cor}\label{sunto}
	If $G$ is insoluble, then:
	\begin{enumerate}[label={(\roman*)},font=\upshape]
		\item $E(G)$ is a non-abelian simple group;  and
		\item $G/E(G)$ is soluble.	\end{enumerate}
\end{cor}
\begin{proof}
  The Frattini subgroup of $G$ is trivial, otherwise, by Proposition \ref{frattini}, $G$ would be nilpotent. This implies that all quasisimple subnormal subgroups of $G$ are simple, and so $\soc(G) = F^*(G) = E(G) \times F(G)$. Moreover $C_G(\soc(G))=F(G),$ and we may write $F(G)=N_1\times \cdots \times N_t$ as a product of abelian minimal normal subgroups of $G.$

  By Lemma \ref{ciclici}, for $1\leq i\leq t$, $N_i\cong C_{p_i}$ for a suitable prime $p_i$ and therefore
  \[G/C_G(F(G)) \cong G/\bigcap_{1\leq i \leq t}C_G(N_i)\leq \prod_{1\leq i \leq t}G/C_G(N_i) \leq \prod_{1\leq i\leq t}\aut(C_{p_i})\]
is abelian. If $E(G) = 1$, then $\soc(G) = F(G)$, so $C_G(F(G))=F(G)$ and $G$ is metabelian, a contradiction.  Hence $G$ contains a non-abelian minimal normal subgroup and so by Lemmas~\ref{semplice} and \ref{unico}, $E(G)$ is a non-abelian simple group $S$. Therefore, 
\[G/F(G)=G/C_G(\soc(G))=G/(C_G(S)\cap C_G(F(G))\leq G/C_G(S) \times G/C_G(F(G)).
\]
Since $G/C_G(S)$ is almost simple with socle isomorphic to $S$, and $G/C_G(F(G))$ is abelian, we conclude that $S = E(G)$ is the unique non-abelian composition factor of $G$ and therefore $G/E(G)$ is soluble.
\end{proof}

We now 
show that in an insoluble group with the independence property, certain elements are dependent, and so in particular must commute.

\begin{lemma}\label{mustcommute}
  Suppose that $G$ is insoluble and let $S=E(G)$ be the unique non-abelian minimal normal subgroup of $G$.
  Suppose in addition that there exist $x, s \in S$
      such that $xC_G(S)$ and $sC_G(S)$ are dependent in  $G/C_G(S)$.
   Then
$x$ and $s$ are dependent, and in particular they commute.
\end{lemma}
\begin{proof}
 Let $\overline G:=G/C_G(S),$ and, for all $g\in G,$ let $\overline g:=gC_G(S).$
 Assume for a contradiction that $\{x,s,g_1,\dots,g_t\}$ is a minimal generating set for $G.$
 We may order the indices so that, for some $r$, the subset $\{ g_1, \dots,  g_r\}$ of $\{ g_1, \dots,  g_t\}$ is minimal, subject to $\overline G=\langle \overline x, \overline s, \overline g_1,\dots, \overline g_r\rangle.$ It follows that there exists $z\in \{x,s\}$ such that $\overline G= \langle   \overline z, \overline g_1,\dots, \overline g_r\rangle.$

 Let
	$R:=\langle z, g_1, \dots, g_t\rangle.$ Clearly
	$G 
        = S\langle z,g_1,\dots,g_t \rangle=RS.$ If $R\cap S=1,$ then $R\cong G/S$ is soluble by Corollary \ref{sunto}, contradicting $RC_G(S)/C_G(S) = \overline G$.
	Thus $R\cap S$ is a non-trivial and normalized by $RC_G(S)=G$, and consequently $S=R\cap S$ and $G=R,$ contradicting the fact that $\{x,s,g_1,\dots,g_t\}$ is a minimal generating set for $G.$
        Hence $x$ and $s$ are dependent, so $\langle x, s\rangle$ is cyclic. 
\end{proof}

Recall that for a group $X$ and an element $x \in X$, we write $\mathcal{M}_X(x)$ for the set of maximal subgroups of $X$ containing $s$. 

\begin{lemma}\label{lem:final_setup}
  Let $X$ be a finite group and let $x, y \in X$ be such that $y \in  \cap_{M \in \mathcal{M}_X(x)} M$.
  Then no minimal generating set for $X$ contains $\{x, y\}$.
\end{lemma}

\begin{proof}
Suppose, for a  contradiction, that there exist $g_1,\dots,g_d\in X$ such  that $\{x,y,g_1,\dots,g_d\}$ is a minimal generating set for $X.$ Then $\langle x, g_1,\dots,g_d\rangle\neq X$, so there exists $M\in \mathcal{M}_X(x)$ such that $\langle x, g_1,\dots,g_d\rangle\leq M.$ By assumption, $y \in M$, so $X=\langle x, y, g_1,\dots,g_d\rangle\leq M,$ a contradiction.
\end{proof}

Finally, we reach the main result of this section. Recall that $G$ is assumed to be a finite group that satisfies the independence property.

\begin{thm} \label{thm:indsupersol}
The group $G$ is supersoluble.
\end{thm}

\begin{proof}
  First assume that $G$ is insoluble, so that $S:=E(G)$ is  non-abelian simple
     by
    Corollary~\ref{sunto}. Let $H \cong G/C_G(S).$ By Theorem~\ref{thm:noncomint},
  $\soc(H) \cong S$ contains two non-commuting elements $xC_G(S)$ and
  $sC_G(S)$ with the property that every maximal subgroup of
  $H$ containing $sC_G(S)$ also contains $xC_G(S)$. 
    Furthermore, since $S \cap C_G(S) = 1$,  we can choose the corresponding 
    $x, s$ to lie in $E(G)$.   By Lemma~\ref{lem:final_setup}, no minimal generating set for $H$ contains $\{xC_G(S), C_G(S)\}$, that is, $xC_G(S)$ and $yC_G(S)$ are dependent in $H$. But then, by Lemma \ref{mustcommute}, $s$ and 
  $x$ must commute, a contradiction. So $G$ is soluble.

  If  $\frat(G) \neq 1$ then $G$ is nilpotent, by Proposition~\ref{frattini}, so assume otherwise.  By Lemma \ref{ciclici}, $F(G)= N_1\times \dots \times N_u$ is a direct product of minimal normal subgroups of prime order, and $G/F(G)\leq \prod_{1\leq i \leq u} \aut(N_i)$ is abelian, hence $G$ is supersoluble.
\end{proof}

\section{Supersoluble groups with the independence property}\label{superind}
In this section we shall determine the structure of the finite supersoluble groups $G$ satisfying the independence property. By Proposition \ref{frattini}, we may restrict our attention to the case $\frat(G)=1$.
This implies that the Fitting subgroup $\fit(G)$ of $G$ is a direct product of minimal normal subgroups of $G,$ and is complemented in $G$. Let $K$ be a complement of $\fit(G)$ in $G.$ Since $G$ is supersoluble,
$G^\prime \leq \fit(G),$ and consequently $K$ is abelian. Let $W$ be a complement of $Z(G)$ in $\fit(G).$ Then $G=W \rtimes H,$ with $H=\langle K, Z(G)\rangle.$ Moreover, by Lemma~\ref{ciclici}, each minimal normal subgroup of $G$ is cyclic (of prime order).
As a consequence, we assume through this section that
\[G=(V_1^{\delta_1} \times \dots \times {V_r^{\delta_r})}\rtimes H,\]
where $H$ is abelian,  $\delta_1, \dots, \delta_r$ are positive integers, and $V_1,\dots,V_r$ are
  irreducible $H$-modules of prime order that are pairwise non-$H$-isomorphic, such that $H$ acts non-trivially on each. (So if $G$ is abelian then $r = 0$ and $G = H$.)
For each $i \in \{1, \ldots, r\}$, 
we set 
\[\begin{array}{c} p_i := |V_i|, \quad
           W_i:=V_i^{\delta_i}, \quad \text{and}
  \quad W:=V_1^{\delta_1} \times \dots \times {V_r^{\delta_r}} = W_1
    \times \dots \times W_r.
  \end{array}         
\]

Observe now that,  for each $h$ in $H$, 
there exists an $\alpha_i(h) \in \mathbb{F}_{p_i}^\times$ such that
$w^h=\alpha_i(h)w$ for all $w\in W_i.$

\begin{notation}\label{not:matrix}
  Let $g_1:= h_1(w_{1,1},\dots,w_{1,r}),\dots, g_t:=h_t(w_{t,1},\dots,w_{t,r})$ be elements of $G$, with $h_i \in H$ and $w_{i,j}:=(x_{i,j,1},\dots,x_{i,j,\delta_j})\in W_j$.  For each $j \in \{1,\ldots,r\}$,  we define a matrix $A^{(j)}=A^{(j)}(g_1,\ldots,g_t)$ whose columns are $a^{(j)}_1,\ldots,a^{(j)}_t$, where for each $i \in \{1,\ldots,t\}$ the transpose of $a^{(j)}_i$ is
\[ 
     ( 1 - \alpha_j(h_i),
      x_{i, j, 1},
      x_{i, j, 2},
      \ldots,
      x_{i, j, \delta_j}) \in \mathbb{F}_{p_j}^{\delta_j + 1}.\]
\end{notation}
\begin{lemma}\label{corone}
  Let $g_1, \ldots, g_t$ be as in Notation~\ref{not:matrix}.
 Then $G = \langle g_1,\dots,g_t\rangle$ if and only if $\langle h_1,\dots,h_t\rangle=H$ and ${\rm{rank}}(A^{(j)})=\delta_{j}+1$ for all $j \in \{1, \ldots, r\}$. 
\end{lemma}
\begin{proof}
  See Propositions 2.1 and 
    2.2 in \cite{cliq}.
\end{proof}

\begin{lemma}\label{lem:three_conditions}
  Suppose that $G$ satisfies the independence property. Then the following statements hold.
	\begin{enumerate}[label={(\roman*)},font=\upshape]
        \item If $\delta_i=1,$ then $|H/C_H(V_i)|$ is prime. 
			\item If $|V_i|=|V_j|$, then $i=j.$
		\item $(|H|,|W|)=1.$
	\end{enumerate}
\end{lemma}
\begin{proof}
   For (i), notice first that by assumption, $|H/C_H(V_i)|> 1$.
    Assume for a contradiction that $\delta_i=1$ but $|H/C_H(V_i)|$ is not prime,  and choose $h\in H$ such that $|hC_H(V_i)|$ is prime. Take $0\neq x\in V_i = W_i$, and let  $g_1:=h= h(0, 0, \ldots, 0)$ and $g_2:=hx = h(0, \ldots, 0, x, 0, \ldots, 0)$ (with $x$ in position $i$). Since $h \not\in C_H(V_i)$, the elements $g_1$ and $g_2$ do not commute,
so neither is a power of the other. Furthermore, $\langle g_1,
  g_2 \rangle \neq G$. Since $G$ satisfies the independence property,
  $g_1$ and $g_2$ are therefore independent, so there exist $g_3,\dots g_t \in G$ such that
$\{g_1,g_2,g_3,\dots,g_t\}$ is a minimal generating set for $G.$
Using Notation~\ref{not:matrix}, since $\langle h_1, h_2, h_3,\dots,h_t\rangle=\langle h, h_3,\dots,h_t\rangle$ and
$\langle h, C_H(V_i)\rangle \neq H,$ there exists $k\geq 3$ such that
$h_k\not\in C_H(V_i).$ Since $\delta_i = 1$,  the matrix
    $A^{(i)}=A^{(i)}(g_1,\ldots,g_t)$ has two rows. From $h_k \not\in C_H(V_i)$ we deduce that
$1-\alpha_i(h_k)\neq 0$, so there exists an $\ell\in \{1,2\}$ such that
 the columns $a^{(i)}_\ell$ and $a^{(i)}_k$ of $A^{(i)}$ are linearly independent. Moreover, if $j\neq i$ then
$a^{(j)}_1 = a^{(j)}_2$.
It follows from Lemma \ref{corone} that $\langle g_\ell, g_3,\dots,g_m\rangle=G,$ a contradiction. This proves (i).

For (ii), assume for a contradiction that $|V_i| = |V_j|$ for distinct $i$ and $j$, and let $a \in V_i \setminus \{0\}$ and $b \in V_j \setminus \{0\}$, so that  $|a| = |b|$. Then $ab$ and $b$ are independent in $G$, and hence Lemma~\ref{due} shows that  there is a $G$-isomorphism, and hence an $H$-isomorphism, between $V_i$ and $V_j$, a contradiction. 
  
For (iii), assume, for a contradiction, that $p_i$ divides $|H|$ for some $i \in \{1,\dots,r\}.$  Fix  $v \in V_i$ and $h\in H$ with $|h|=|v|=p_i.$
Neither of $g_1:=h$ and $g_2:=hv$ is a power of the
other, so there exist $g_3,\dots,g_t \in G$ such that
$\{g_1,g_2,g_3,\dots,g_t\}$ is a minimal generating set for $G.$ The
condition $|h|=|v|=p_i$ implies that $h \in C_H(V_i)$. With the notation of Notation~\ref{not:matrix}, we notice that $\langle h_1, h_2,
h_3,\dots,h_t\rangle=\langle h, h_3,\dots,h_t\rangle$, and if $j \neq i$ then
$a^{(j)}_1 = a^{(j)}_2$, and $a^{(i)}_1$
is zero. But then  $\langle g_2,\dots,g_t\rangle=G,$ a
contradiction. 
\end{proof}

The following example shows that three necessary conditions in Lemma~\ref{lem:three_conditions} are not sufficient to ensure that $G$ satisfies the independence property.

\begin{eg}
Consider the group \[G:=(V_1\times V_2^2) \rtimes (\langle x \rangle \times \langle y \rangle) \cong \AGL_1(3) \times (\FF_5^2 \rtimes \langle 2I_2 \rangle) \le \AGL_1(3) \times \AGL_2(5),\] so that $|V_1|=3,
|V_2|=5, |x|=4, |y|=2$, $x \in C_G(V_1),$ $y\in C_G(V_2),$ $w^x=2w$
for all $w\in V_2^2$ and $v^y=2v$ for all $v\in V_1.$ Observe that $G$ satisfies Conditions (i)--(iii) of Lemma~\ref{lem:three_conditions}.

We claim that $x^2y$ and $y$ are dependent in $G$. To see this, suppose for a contradiction that $\{g_1,\ldots,g_r,x^2y,y\}$ is a minimal generating set for $G$, and let $R:=\langle g_1,\ldots,g_r,x^2y \rangle$. It is not restrictive to assume that $g_1,\ldots,g_r \in (V_1 \times V_2^2)\langle x \rangle$. Notice that $R V_1 \langle y \rangle = G$, and therefore $R$ contains $xy^kv$ for some $k \in \mathbb{Z}$ and $v \in V_1$. However, $(xy^kv)^2 = x^2$, so $y \in R$, a contradiction.
As $x^2y$ and $y$ generate distinct subgroups of order $2$, the group $G$ does not satisfy the independence property.
\end{eg}

In the remainder of the section, we will determine some additional  conditions. For this purpose we need some definitions.
Let $g_1$ and $g_2$ be as in Notation~\ref{not:matrix}.  Then we say that $g_1$ and $g_2$ are \emph{$j$-independent} for some $j \in \{1, \ldots, r\}$ if either $\mathrm{rank}(A^{(j)}(g_1,g_2)) = 2$
  or, for some ordering $\lambda,\mu$ of $1$ and $2$, the column $a^{(j)}_\lambda \neq 0$, $a^{(j)}_\mu = 0$, and $\langle h_\lambda \rangle \frat(H) \neq \langle h_1, h_2 \rangle \frat(H).$


Similarly, $g_1$ and $g_2$ are \emph{$\{i,j\}$-independent} for distinct $i,j \in \{1,\ldots,r\}$ if $\langle h_1 \rangle\frat(H) = \langle h_2 \rangle \frat(H)$ and  for some ordering $\lambda,\mu$ of $1$ and $2$,  the columns $a^{(i)}_\lambda$ and $a^{(j)}_\mu$ are non-zero, and $a^{(j)}_\lambda = a^{(i)}_\mu = 0$.

\begin{lemma}\label{crit}
Assume that $G$ satisfies the three conditions of Lemma~\ref{lem:three_conditions}.
Two elements $g_1=h_1(w_{1,1},\dots,w_{1,r})$ and $g_2=h_2(w_{2,1},\dots,w_{2,r})$ are independent in $G$ if and only if one of the following statements holds:
\begin{enumerate}[label={(\roman*)},font=\upshape]
\item $h_1$ and $h_2$ are independent in $H;$
\item  there exists $j \in \{1,\ldots,r\}$ such that
$g_1$ and $g_2$ are $j$-independent; or
\item there exist distinct $i,j \in \{1,\ldots,r\}$ such that $g_1$ and $g_2$ are $\{i,j\}$-independent.
\end{enumerate}
\end{lemma}
\begin{proof}
If $G$ is abelian, i.e., if $G = H$, then (ii) and (iii) cannot hold, and (i) is equivalent to $g_1$ and $g_2$ being independent. Thus we will assume that $G$ is non-abelian. We first show the sufficiency of each of the three conditions in turn, and then the necessity that at least one of them holds. While proving the sufficiency of each condition, we may assume that $G$ is not equal to $\langle g_1, g_2 \rangle$, as otherwise $g_1$ and $g_2$ are independent.
  
  If $h_1$ and $h_2$ are independent in $H$, then for some $a \ge 2$ there exist $h_3,\dots,h_a$ in $H$ such that $\{h_1,h_2,h_3,\dots,h_a\}$ is a minimal generating set for $H.$ Then we take $v_1,\dots,v_b \in W$ such that $\{v_1,\dots,v_b\}$ is minimal, subject to \[\{g_1,g_2,h_3,\dots,h_a,v_1,\dots,v_b\}\] being a generating set for $G.$ In this way we obtain a minimal generating set for $G$ containing $g_1$ and $g_2.$

  \smallskip

Next suppose that $g_1$ and $g_2$ are $j$-independent for a fixed $j$. 
First,  let $A := A^{(j)}(g_1, g_2)$, and assume that ${\rm{rank}}( A)=2$.
Let $\{h_3,\dots,h_m\}$ be a subset of $H$ of minimal cardinality, subject to $m\geq \delta_j+1$ and
$H=\langle h_1, h_2, h_3,\dots, h_m \rangle.$ Since $H$ acts as scalars on $V_j$, the group  $H/C_H(V_j)$ is cyclic,  so
we may choose $h_3,\dots,h_m$ 
so that $\langle h_1, h_2, h_3, C_H(V_j)\rangle = H$ and $h_i\in C_H(V_j)$ if $i\geq 4$. If $\langle h_1, h_2, C_H(V_j)\rangle = H$ and $m\geq 3,$ we can additionally require $h_3\in C_H(V_j).$
 Notice in particular that if $\delta_j=1$, then $H/C_H(V_j)$ is assumed to be of prime order, so $\langle h_1,h_2,C_H(V_j)\rangle =H$ and  either $m=2$ or $h_3 \in C_H(V_j)$.
If $m > 2$, then choose $z_{i,k} \in \mathbb{F}_{p_j}$ for $i \in \{3, \ldots,  \delta_j+1\}$ and $k \in \{1, \ldots,  \delta_j\}$  such that
\[{\rm{rank}}\begin{pmatrix}
	1-\alpha_j(h_1)&1-\alpha_j(h_2)&1-\alpha_j(h_3)&0&\dots&0\\
	x_{1,j,1}&x_{2,j,1}&z_{3,1}&z_{4,1}&\dots&z_{\delta_{j+1},1}\\
	\vdots&\vdots&\vdots&\vdots&\cdots&\vdots\\
	x_{1,j,\delta_j}&x_{2,j,\delta_j}&z_{3,\delta_j}&z_{4,\delta_j}&\dots&z_{\delta_{j+1},\delta_j}
\end{pmatrix}=\delta_j+1,
\]
 which is possible since $\mathrm{rank}(A) = 2$.
 Set $g_i:=h_i(z_{i,1},\dots,z_{i,\delta_j})$ for each $i \in \{3,\ldots,\delta_j+1\}$, and $g_i:=h_i$ for each $i \in \{\delta_j+2,\ldots,m\}$. Additionally, let \[Z_j:=\underset{i \in \{1, \ldots, r\} \setminus \{j\}}
   {\prod}W_i.\]
Then, by Lemma \ref{corone}, $\{g_1Z_j,g_2Z_j,g_3Z_j,\dots,g_mZ_j\}$ is a minimal generating set for $G/Z_j$ and  therefore there exist elements $g_{m+1}, \ldots, g_{m+\ell}$ of $Z_j$, for some $\ell$, such that $\{g_1,g_2,g_3,\dots,g_{m+ \ell}\}$
is a minimal generating set for $G$.

 Hence without loss of generality we may assume that
  $a^{(j)}_1 \neq 0$, $a^{(j)}_2=0$ and $\langle h_1 \rangle \frat(H) \neq \langle h_1, h_2\rangle \frat(H).$
There exists a (possibly empty) subset  $\{h_3,\dots,h_\ell\}$ of $H$ such that
$\{h_2, h_3, \dots, h_\ell\}$ is minimal subject to $\langle h_1, h_2, h_3, \dots, h_\ell\rangle=H.$
Since $H/C_H(V_j)$ is cyclic,
we may choose $h_4,\dots,h_\ell$ to lie in $C_H(V_j)$.
Let $m:=\max 
\{\ell,\delta_j+2\}$ and set $h_i=1$ for each $i \in \{\ell + 1, \ldots, m\}$. 
 If $\alpha_j(h_1) = 1$ then $\alpha_j(h_3) \neq 1$, since $H$ does not centralise $V_j$. Thus we may choose $z_{i,k} \in \mathbb{F}_{p_j}$ for $i \in \{3, \ldots, \delta_j+2\}$ and $k \in \{1, \ldots,  \delta_j\}$ such that
\[{\rm{rank}}\begin{pmatrix}
	1-\alpha_j(h_1)&1-\alpha_j(h_3)&0&\dots&0\\
x_{1,j,1}&z_{3,1}&z_{4,1}&\dots&z_{\delta_{j+2},1}\\
	\vdots&\vdots&\vdots&\cdots&\vdots\\
x_{1,j,\delta_j}&z_{3,\delta_j}&z_{4,\delta_j}&\dots&z_{\delta_{j+2},\delta_j}
\end{pmatrix}=\delta_j+1.
\]
 Set $g_i:=h_i(z_{i,1},\dots,z_{i,\delta_j})$  for each $i \in \{3,\ldots,\delta_j+2\}$, and $g_i:=h_i$ for each $i \in \{\delta_j+3,\ldots,m\}$. Then
 $\{g_1Z_j,g_2Z_j,g_3Z_j,\dots,g_mZ_j\}$ is a minimal generating set for $G/Z_j$ and as in the previous case, there exist elements of $Z_j$ that extend $\{g_1,g_2,g_3,\dots,g_m\}$ to a minimal generating set for $G$.

   \smallskip
   
 Finally, for this direction of the proof, assume that $g_1$ and $g_2$ are $\{i, j\}$-independent for distinct $i$ and $j$, so that $\langle h_1 \rangle\frat(H) = \langle h_2 \rangle \frat(H)$. Without loss of generality, $a^{(j)}_1$ and $a^{(i)}_2$ are non-zero, whilst  $a^{(j)}_2 = a^{(i)}_1=0.$
Let $B := \{h_3,\dots,h_m\}$ be a subset of $H$ of minimal cardinality subject to $m\geq \max\{\delta_i,\delta_j\}+2$ and
$H=\langle h_1, B \rangle=\langle h_2, B \rangle.$  Without loss of generality, suppose that $\delta_i\leq \delta_j$.  Let $C := C_G(V_i) \cap C_G(V_j)$. We shall split into two cases to place various assumptions on $B$  and define associated
  matrices over $\mathbb{F}_{p_i}$ and $\mathbb{F}_{p_j}$, before concluding both cases of the proof.

\smallskip

\paragraph{\bf Case (a):} $\delta_j \ge 2$, or $\delta_i = \delta_j = 1$ and either $G/C$ is cyclic or $\langle h_1,h_2\rangle\not\leq C$. Since $H/C_H(V_j)$ and $H/C_H(V_i)$ are cyclic,
if $\delta_j\geq 2$ then we may  assume that $h_4,\dots,h_m$ lie in $C_H(V_i)$, and satisfy $h_k \in C \cap H$ if $k>4.$  Similarly, if $\delta_i=\delta_j=1$, then $H/(C \cap H) \cong G/C$ is either cyclic or isomorphic to $C_q^2$ for some prime $q$ (this is because $\langle h_1,h_2\rangle\not\leq C$ is equivalent to $h_1,h_2 \notin C$), so we may assume 
  in this case that $h_4,\dots,h_m \in C \cap H$.
  In all three cases, we may choose elements
  $z_{\ell,s,k} \in \mathbb{F}_{p_s}$ for $\ell \in \{3,  \ldots, m\}$, $s \in \{i,j\}$ and $k \in \{1, \ldots, \delta_s\}$ such that 
\[{\rm{rank}}\begin{pmatrix}
1-\alpha_s(h_{f(s)})&1-\alpha_s(h_3)&\dots&1-\alpha_s(h_{\delta_{s}+2})\\
x_{f(s),s,1}&z_{3,s,1}&\cdots&z_{\delta_{s}+2,s,1}\\
\vdots&\vdots&\cdots&\vdots\\
x_{f(s),s,\delta_s}&z_{3,s,\delta_s}&\cdots&z_{\delta_s+2,s,\delta_s}
\end{pmatrix}=\delta_s+1\]
for each $s$, where $f(j):=1$ and $f(i):=2$, and $z_{\ell, s, k} = 0$  $\ell >\delta_s + 2$.

\medskip

\paragraph{\bf Case (b):} $\delta_i=\delta_j=1$, $G/C$ is not cyclic, and $\langle h_1,h_2\rangle \leq C$. Here, our assumptions on $G$ imply that $H/(C \cap H)\cong C_q^2$ for some prime $q$. Hence we may assume that
 $h_3 \in C_H(V_i) \setminus C_H(V_j)$,  $h_4 \in C_H(V_j) \setminus C_H(V_i)$, and $h_5, \ldots, h_m \in C \cap H$.
Since $h_1,h_2 \in C$, and $a^{(j)}_1$ and $a^{(i)}_2$ are non-zero by assumption, we observe that \[{\rm{rank}}\begin{pmatrix}
	1-\alpha_j(h_1)&1-\alpha_j(h_3)\\
	x_{1,j,1}& 0
\end{pmatrix} = {\rm{rank}}\begin{pmatrix}
1-\alpha_i(h_2)&1-\alpha_i(h_4)\\
x_{2,i,1}&  0
\end{pmatrix}=2.
\] 
We set $z_{k,s,1} := 0$ for all $k \in \{3,\ldots,m\}$ and $s \in \{i,j\}$.

\smallskip

To conclude Cases (a) and (b), we set \[g_k:=h_k(z_{k,j,1},\dots,z_{k,j,\delta_j})(z_{k,i,1},\dots,z_{k,i,\delta_i})\] for each $k \in \{3, \ldots, m\}$, and let  $Z_{i,j}:=\prod_{k\neq i, j}W_k$.
Then $G/Z_{i,j}$ has minimal generating set $\{g_1Z_{i,j},g_2Z_{i,j},g_3Z_{i,j},\dots,g_mZ_{i,j}\}$, and therefore  there exist elements of $Z_{i, j}$ that extend $\{g_1,\dots,g_m\}$ to a minimal generating set for $G$, as required.

\medskip

For the converse direction, suppose that $h_1$, and $h_2$ are dependent in $H$, and that there is no $j\in \{1,\dots, r\}$ such that $g_1$ and $g_2$ are $j$-independent and no 2-set $\{i,j\}$ such that $g_1$ and $g_2$ are $\{i,j\}$-independent.  
Assume, for a  contradiction, that $\{g_1,g_2,g_3,\dots,g_d\}$ is a minimal generating set for $G$, with $d \ge 2$. We may assume that $h_2=h_1^tf$ with $t\in \mathbb Z$ and $f\in \frat(H),$ so in particular $H=\langle h_1, h_3, \dots, h_{d}\rangle.$

For each $j$, let $A^{(j)}:= A^{(j)}(g_1,\ldots,g_d)$ and let $B^{(j)}$
be
the matrix obtained from $A^{(j)}$ by deleting its second column. 
Since $g_1$ and $g_2$ are $j$-dependent, columns $a^{(j)}_1$ and $a^{(j)}_2$ are linearly dependent.  If $a^{(j)}_2=0$ whenever  $a^{(j)}_1=0$, then  ${\mathrm{rank}}(A^{(j)})={\mathrm{rank}}(B^{(j)})$ for all $j$, and therefore $G=\langle g_1, g_3, \dots, g_{d}\rangle$ by Lemma~\ref{corone}, a contradiction. So there exists a $k$ such that  $a^{(k)}_1=0$ and $a^{(k)}_2\neq 0$. Since $g_1$ and $g_2$ are $k$-dependent, $\langle h_2 \rangle \frat(H) = \langle h_1, h_2\rangle \frat(H)=\langle h_1, h_1^tf\rangle \frat(H)=\langle h_1 \rangle \frat(H).$ In particular $H=\langle h_2, h_3, \dots, h_d\rangle$ and as before we can conclude that there exists an $\ell$ such that $a^{(\ell)}_2=0$ and $a^{(\ell)}_1\neq 0$. But then $g_1$ and $g_2$ are $\{k,\ell\}$-independent, a contradiction.
\end{proof}

In what follows, we let $I_h:=\{i \in \{1,\dots,r\} \mid h \in C_H(V_i)\}$ for $h \in H$, 
    and set $F:=\frat(H)$.

\begin{cor}\label{coropiu}
   Assume that $G$ satisfies the three conditions of Lemma~\ref{lem:three_conditions}, and let $x, y \in H$ be such that
  $y \in \langle x \rangle F$
   and 
   $I_x \subseteq I_y$. Then $x$ and $y$ are dependent in $G$.  If, in addition, $G$ satisfies the independence property and $I_x \neq \emptyset,$ then $y \in \langle x\rangle$.
\end{cor}	
	
\begin{proof} 
  If $G = H$, then 
    the result is clear, so assume that $r \ge 1$. 
For  each $j \in \{1,\ldots,r\}$, let $A^{(j)}:=A^{(j)}(x,y)$. Clearly $ \mathrm{rank}(A^{(j)}) \leq 1$ for every $j.$	Notice that $a^{(j)}_1=0$ if and only if $x \in C_H(V_j)$, and $a^{(j)}_2 = 0$ if and only if $y \in C_H(V_j)$. In particular, since
$I_x \subseteq I_y$, if $a^{(j)}_1=0$ then $a^{(j)}_2=0$. This, together with  $\langle x, y\rangle F=\langle x\rangle F,$ implies that $x$ and $y$ are $j$-dependent for all $j$, 
and
$\{j_1,j_2\}$-dependent for all 2-subsets $\{j_1, j_2\}$.   
The result now follows from Lemma \ref{crit}.

Assume now in addition that $G$ satisfies the independence property and that $I_x\neq \emptyset.$ Let $k \in I_x\cap I_y$, $0\neq v \in  V_k$, 
 and for $j \in \{1, \ldots, r\}$ let $\tilde A^{(j)}:=A^{(j)}(xv,y)$, with columns $\tilde
  a^{(j)}_1$ and $\tilde a^{(j)}_2$. Then $\tilde a^{(k)}_1 \ne 0$, $\tilde a^{(k)}_2=0$, and $\tilde
  A^{(j)}=A^{(j)}$ if  $j\neq k$. Thus,
arguing as before, we conclude that
 $xv$ and $y$ are dependent, and consequently one of $xv$ and $y$ is a power of the other. As $H \cap V_k$ is trivial, we deduce that $y \in \langle xv \rangle$, and in fact $y \in \langle x \rangle$.
\end{proof}

Thus if $G$ satisfies the independence property, then Conditions (a)--(d) of Theorem~\ref{thm:supersol_struc}\ref{supersol3} hold. The following result completes the proof of Theorem~\ref{thm:supersol_struc}.

%


\begin{prop}
 Assume that $G$ satisfies Conditions (a)--(d) of Theorem~\ref{thm:supersol_struc}\ref{supersol3}. Then $G$ satisfies the independence property.
\end{prop}
	
\begin{proof} Let $g_1:=h_1w_1$ and $g_2:=h_2w_2$ be elements, 
as in Notation~\ref{not:matrix}, that are dependent in $G$. 
  We shall prove that one of $g_1$ and $g_2$ is a power of the other.

   By Lemma~\ref{crit}, the fact that $g_1$ and $g_2$ are dependent in $G$ implies that
$h_1$ and $h_2$ are dependent in $H$. We may therefore assume
throughout the proof that $h_2=h_1^uf$ for some $u \in
\mathbb Z$ and $f\in F$.   If $r = 0$ then $G = H$, and so $F = 1$ and $g_2=g_1^u$, as required. Hence we may assume that $r \ge 1$. 
  
For $1\leq j \leq r,$ consider the matrix $A^{(j)}:=A^{(j)}(g_1,g_2)$. We shall repeatedly use the fact that \[\mathrm{rank}(A^{(j)}) < 2 \mbox{ for all } j \in \{1, \ldots, r\},\]
since otherwise $g_1$ and $g_2$ are $j$-independent, and hence by Lemma~\ref{crit} are independent, a contradiction. Note that that lemma's proof uses  Condition~(a).

By conjugating $g_1$ and $g_2$ by a common element of $W$, if necessary,  we may assume that
$w_{1,j}=0$ whenever $j \notin I_{h_1}$. Hence if
$j \notin I_{h_1}$ then $w_{2,j}=0,$ and otherwise $\mathrm{rank}(A^{(j)})=2$. Thus if $I_{h_1} = \emptyset$, then $g_1 = h_1$ and $g_2 = h_2$, and the result follows from Condition~(d). If $I_{h_2} = \emptyset$, then we reach the same conclusion by a similar argument, corresponding to conjugation by a different element of $W$. Therefore, we shall assume for the remainder of the proof that $I_{h_1}, I_{h_2} \neq \emptyset$, and that $w_{1,j} = w_{2,j} = 0$ whenever $j \notin I_{h_1}$. We distinguish between two possibilities.

\smallskip

\paragraph{\bf Case (a):} $\langle h_2\rangle F\neq \langle h_1, h_2 \rangle F$. Since $g_1$ and $g_2$ are $j$-dependent for all $j \in \{1, \ldots, r\}$, if
 $a^{(j)}_1=0$ then $a^{(j)}_2=0$. Fix $k\in I_{h_1}$, so that $a^{(k)}_1=(0,w_{1,k})^T.$ If $w_{1,k}=0,$ then $a^{(k)}_1=0$ and consequently $a^{(k)}_2=0$, yielding $k \in I_{h_2}.$ 
If  $w_{1,k}\neq 0$ and $k \notin I_{h_2},$ then $\mathrm{rank}(A^{(k)})=2,$
a contradiction. So $I_{h_1}\subseteq I_{h_2}$.

Letting $\mathcal{C}_1:=\{j \in I_{h_1} \mid w_{1,j}\neq 0\}$
and $\mathcal{C}_2:=\{j \in  I_{h_1} \mid w_{2,j}\neq 0\}$,
we see that $\mathcal{C}_2\subseteq \mathcal{C}_1$, $w_1=\prod_{j\in  \mathcal{C}_1}w_{1,j}$, and
$w_2=\prod_{j\in \mathcal{C}_2}w_{2,j}$. Since $|V_1|,\dots, |V_t|$
are pairwise coprime by Condition~(b), and since
$\mathrm{rank}(A^{(j)}) < 2$ for each $j$, it follows that
$w_2=w_1^\ell$ for some $\ell\in \mathbb Z$.  We also deduce from  Condition~(d) that  $h_2=h_1^t$ for some $t\in \mathbb Z$. Since $(|W|,|H|)=1$  by Condition~(c),
there exists $s\in \mathbb Z$ such that $s\equiv t \mod |H|$ and $s\equiv \ell \mod |W|$, and so $g_2=w_2h_2=w_1^\ell h_1^t=w_1^sh_1^s=g_1^s.$

\smallskip

\paragraph{\bf Case (b):} $\langle h_1\rangle F = \langle h_2 \rangle F,$
so that  $h_1 \in \langle h_2\rangle F$ and $h_2 \in \langle h_1\rangle F$.
Assume first that 
there exist $j \in I_{h_1}\setminus I_{h_2}$ and  $k \in I_{h_2}\setminus I_{h_1}$. If  $a^{(j)}_1\neq 0$ then  $\mathrm{rank}(A^{(j)})=2$, a contradiction. So $a^{(j)}_1=0$, and similarly, $a^{(k)}_2=0$. But then
$g_1,g_2$ are $\{j, k\}$-independent, and therefore, by Lemma \ref{crit}, independent in $G$, a contradiction.

We may therefore assume (by swapping $g_1$ and $g_2$ throughout the proof if necessary) that $\emptyset \neq I_{h_1}\subseteq I_{h_2}$, 
so that by Condition~(d), $h_2=h_1^t$ for some $t\in \mathbb Z.$
Assume that $\Omega:=I_{h_2}\setminus I_{h_1} \neq \emptyset.$ For each $\omega \in \Omega,$ we see that $a^{(\omega)}_1\neq 0$ and $a^{(\omega)}_2=0$. 
Therefore
if $k \in I_{h_1}$ and $w_{2,k}\neq 0,$ then $w_{1,k}\neq 0$
(otherwise $a^{(k)}_2$ is non-zero, $a^{(k)}_1=0$, and $g_1$ and $g_2$ are
$\{\omega,k\}$-independent, a contradiction). We deduce that $w_2=w_1^\ell$ for some $\ell\in \mathbb Z$, and as in Case (a), that
$g_2 \in \langle g_1\rangle.$

So we may assume that $I_{h_1}=I_{h_2}\neq \emptyset$. It follows from  two applications of Condition~(d) that $\langle h_1 \rangle =\langle h_2 \rangle$, $g_1=h_1(\prod_{j \in I_{h_1}}w_{1,j})$, and $g_2=h_2(\prod_{j \in I_{h_1}}w_{2,j})$. Let $\mathcal{C}_1:=\{j\in I_{h_1} \mid w_{1,j}\neq 0\}$
and $\mathcal{C}_2:=\{j\in I_{h_1} \mid w_{2,j}\neq 0\}$. If neither  $\mathcal{C}_1 \subseteq \mathcal{C}_2$ nor $\mathcal{C}_2 \subseteq \mathcal{C}_1$, 
then $g_1$ and $g_2$ are $\{i,k\}$-independent for all $i \in \mathcal{C}_1 \setminus \mathcal{C}_2$ and $k \in \mathcal{C}_2 \setminus \mathcal{C}_1$, a contradiction. So without loss of generality $\mathcal{C}_2\subseteq \mathcal{C}_1$, and we conclude again that $g_2$ is a power of $g_1.$
\end{proof}


\section{Groups satisfying the rank-independence property}\label{rankperf}

The aim of this section is to classify the  finite groups that satisfy the rank-independence property, so assume throughout this section that $G$ is a finite group. Note that the proof of the following result uses the classification of finite simple groups.

\begin{thm}[\cite{gene}]\label{minnorm}
  If  $N$ is a minimal normal subgroup of $G$ with $N \neq G,$
	then $d(G) \leq d(G/N) + 1.$
      \end{thm}

  We can easily reduce to the supersoluble case. In the next proof, $\fit(G)$ denotes the Fitting subgroup of $G$. 
\begin{prop}
	Assume that $G$ satisfies the rank-independence property. Then $G$ is supersoluble.
\end{prop}

\begin{proof}
 If $G$ is cyclic then the result is immediate, so assume otherwise.
 First we claim that $G/\frat(G)$ is not simple, so assume otherwise for a contradiction.
Since $G$ is not cyclic,  $G/\frat(G)$ is a non-abelian simple group and therefore it contains two distinct involutions
$x\frat(G)$ and $y\frat(G)$. But then $\langle x, y \rangle$ is not cyclic, so $x$ and $y$ are rank-independent.  However, $\langle x, y  \rangle \frat(G)/\frat(G)$ is dihedral, so in particular $\langle x, y \rangle$ is a proper subgroup of $G$, contradicting the fact that $d(G)=d(G/\frat(G))=2$. Therefore $G/\frat(G)$ is not simple.

	Now 	let $N/\frat(G)$ be a minimal normal subgroup of $G/\frat(G)$.  By Theorem \ref{minnorm}, $d(G)=d(G/\frat(G))\leq d(G/N)+1.$ In particular, no
	two distinct elements of  $N$ are rank-independent, so every pair of such elements generates a cyclic group. It follows that no generating set for $N$ of minimal size contains more than one element, i.e., $N$ is cyclic. So $\fit(G)/\frat(G)$ is a direct product of cyclic minimal normal subgroups of $G/\frat(G)$ and this implies that $G/\frat(G)$ is supersoluble, and hence $G$ is supersoluble.
\end{proof}


\begin{lemma}\label{sylow}Assume that $G$ satisfies the rank-independence property, and has a  non-cyclic normal
  $p$-subgroup $P$. Then $P$ is either
  elementary abelian or generalized quaternion.
\end{lemma}

\begin{proof}
 If $\frat(P) = 1$, then $P$ is elementary abelian, and the result is immediate, so let $x$  be an element of $\frat(P)$ of order $p.$ Since $\frat(P) \leq \frat(G),$ the element $x$ does not lie in any minimal generating set for $G$, and so $\langle x, y \rangle$ is cyclic for all $y \in G$. 
This implies in particular that $\langle x\rangle$ is the unique minimal subgroup of $P.$ Since we are assuming that $P$ is not cyclic, it follows from \cite[Thm.~9.7.3]{wrscott} that $P$ is a generalized quaternion group.
\end{proof}

\begin{prop}\label{nilp}
	Assume that $G$ is a non-cyclic nilpotent group. Then $G$ satisfies the rank-independence property  \ifa one of the following holds:
	\begin{enumerate}[label={(\roman*)},font=\upshape]
		\item $G\cong C_p\times C_p$;
		\item $G\cong Q_8$;
		\item $G\cong P\times C$, with $P$ an elementary abelian Sylow subgroup of $G$ \st $d(P)\geq 3$, and $C$ cyclic (this includes the case $|C| = 1$).
	\end{enumerate}
\end{prop}
\begin{proof}
	It is easy to check that if $G$ satisfies (i), (ii) or (iii), then $G$ satisfies the rank-independence property. Conversely,
        suppose that $G$ has the rank-independence property and let $d:=d(G).$ There exists a prime $p$ dividing $|G|$ such that the Sylow $p$-subgroup $P$ of $G$ satisfies $d(P)=d$.
        Let $C$ be a $p$-complement in $G.$ Distinct elements of $C$ cannot belong to a generating set for $G$ of cardinality $d$, so $C$ is cyclic. If $d\geq 3,$ then $P$ is an elementary abelian $p$-group by Lemma \ref{sylow}. If $d=2$, then each proper subgroup of $G$ is cyclic, hence $G=P$ is isomorphic either to $C_p\times C_p$ or to $Q_8.$
\end{proof}

\begin{prop}\label{PT}
Assume that $G$ is not nilpotent, and that $d(G)\geq 3$. Then $G$ satisfies the rank-independence property if and only if $G=P\rtimes C$,  where $P$ is an elementary abelian Sylow $p$-subgroup of $G$
	and $C$ is a cyclic group, 
        acting on $P$ as scalar multiplication. 
\end{prop}

\begin{proof}Assume that $G$ satisfies the rank-independence property and set $F:=\frat(G).$ Since $G$ is supersoluble, 
\begin{equation*}
	G/F \cong (V_1^{\delta_1}\times \dots \times V_r^{\delta_r})\rtimes H
\end{equation*}
where $H$ is abelian, $\delta_1, \dots, \delta_r$ are positive integers, and $V_1,\dots,V_r$  are pairwise non-$H$-isomorphic, irreducible $H$-modules on each of which $H$ acts non-trivially. Moreover,  $r > 0$ since $G$ is not nilpotent, and for $i \in \{1, \ldots, V_i\}$ the group $V_i$ has prime order $p_i$.
By Lemma~\ref{corone},
\begin{equation}\label{d(G)}
	d:=d(G)=\max \{d(H), \delta_i+1\mid 1\leq i\leq r\}.
\end{equation}
Since $r \neq 0$, there exists a non-central minimal normal subgroup $N/F$ of $G/F.$ Since $N$ is non-central, there exist $x\in N$ and $y\in  G$ such that 
$[x,y]\neq 1.$  Now $\langle x, y \rangle$ is not abelian, so there exist $z_3,\dots,z_d \in G$ such that $G=\langle x, y, z_3,\dots,z_d\rangle.$ In particular, $ G=N\langle  y, z_3,\dots, z_d\rangle$ and hence $d(H)\leq d(G/N)\leq d-1$. It follows from \eqref{d(G)}  that $d=\delta_i+1$ for some $i\in \{1,\dots,r\}$. We may assume that $i=1.$

Set $V:=V_1$, $p:=p_1,$ and $\delta := \delta_1=d-1.$
We identify each element $w$ of $V^{{\delta}}$ with an element
$(x_1,\dots,x_{\delta})\in \FF_p^{{\delta}}$. Let $L=C/F$ be  a complement of $V^{{\delta}}$ in $ G/F.$ For each $\ell$ in $L,$  there exists an $\alpha(\ell) \in  \FF_p^\times$ such that
$w^\ell=\alpha (\ell)w$ for all  $w\in V^{{\delta}}.$ 
Given $g_1,\dots,g_d\in G,$ we shall write $g_iF=\ell_i(x_{i,1},\dots,x_{i,{\delta}})$, with
$x_{i,j} \in \FF_p$.
  Consider the matrix
	\[A = A(g_1,\dots,g_d):=\begin{pmatrix}
	1-\alpha(\ell_1)&\cdots&1-\alpha(\ell_d)\\
	x_{1,1}&\cdots&x_{d,1}\\
	\vdots&\cdots&\vdots\\
	x_{1,{\delta}}&\cdots&x_{d,{\delta}}
	\end{pmatrix},
	\]
        similar to Notation~\ref{not:matrix}. 
It follows from Lemma \ref{corone} that $\langle g_1,\dots,g_d\rangle=G$ if and only if 
$\langle \ell_1,\dots,\ell_d\rangle=L$, and ${\rm{rank}}(A)=d.$

We now show that $C$ is cyclic. Assume for a  contradiction that 
there exist $g_1, g_2 \in C$ such that $\langle  g_1,  g_2\rangle $ is not cyclic. 
Since $G$ satisfies the rank-independence property,  there exist $g_3,\dots,g_d\in G$ such that $\langle g_1,\dots,g_d\rangle=G$. However
the  first two columns of the matrix $A = A(g_1, \ldots, g_d)$ are linearly dependent, contradicting $\mathrm{rank}(A) = d$.
Since $C$ is cyclic, $L$ is also cyclic  and consequently $r=1,$ and so \[G/F \cong V^{{\delta}}\rtimes H.\]

We show next that $p$ does not divide $|H|.$ Assume, for a contradiction, that there exists $y_1 \in G$ such that $y_1F$ is an element of $H$ of order $p$, and choose $y_2 \in G$ so that $y_2F$ is a non-trivial element of $V^{{\delta}}.$ Then $K=\langle y_1, y_2\rangle F \cong (C_p \times C_p)F$ is a non-cyclic normal subgroup of $G$. As $G$ has the rank-independence property, there exist $y_3,\dots,y_d \in G$ such that $G=\langle y_1, y_2, y_3,\dots, y_d\rangle$. However, this is not possible, since  $d(G/K) = d( V^{{\delta}-1} \rtimes H/\langle y_1\rangle) > d-2$. Hence $(p, |H|) = 1$.

Let $P$ be a Sylow $p$-subgroup of $G$. From $(p, |H|) = 1$ we deduce
that $P\cap C\leq F.$ Moreover, $P$ is a normal subgroup of $G$ and $d(P)\geq {\delta}=d-1\geq 2,$ so, by Lemma \ref{sylow}, $P$ is either  elementary abelian or generalized quaternion. But in the second case, $|V|=2$,  contradicting the assumption that $H$ acts non-trivially on $V$. Thus $P$ is elementary abelian, and consequently $P\cap C\leq P\cap F=1$. In particular $P\cong V^{\delta}$ and $G\cong V^{\delta}\rtimes C$, so $p$ does not divide $C$.

\medskip

Conversely, assume  $G=P \rtimes C$, with $P\cong C_p^{{\delta}}$, ${\delta}\geq 2,$ $(|C|,p)=1$ and $C$ acting on $P$ as scalar multiplication. Since $G$ is not nilpotent, $C$  acts non-trivially on $P$. Additionally, $d:=d(G)={\delta}+1.$ We again apply Lemma \ref{corone}. We identify each element of $P$ with a vector $(y_1,\dots,y_{\delta}) \in \FF_p^{\delta}$.
For each $c\in C,$ there exists an $\alpha(c)\in \FF_p^\times$ such that $y^c=\alpha(c)y$ for all $y\in P.$ Let $g_1:=  c_1(y_{1,1},\dots,y_{1,{\delta}}), g_2:= c_2(y_{2,1},\dots,y_{2,{\delta}}) \in G$, and
suppose $\langle x \rangle=C$.  If $\langle g_1, g_2\rangle$ is not cyclic, then we may  choose $y_3, y_4, \ldots, y_d \in \FF_p^{\delta}$
in such a way that $\mathrm{rank}(A(g_1, g_2, xy_3, y_4, \ldots, y_d))$ is equal to $d$.
Then by Lemma~\ref{corone}, $\langle g_1, g_2, xy_3,\dots,y_{d}\rangle = G$, as required.
	\end{proof}

Propositions~\ref{nilp} and \ref{PT} combine to prove Theorem~\ref{thm:rank_perfect3}. The following result completes the proof of Theorem~\ref{thm:rank_perfect2}.

\begin{prop}
  Assume that $G$ is not nilpotent, and that $d(G)=2$. Then $G$ satisfies the rank-independence property if and only if $G$ is as described in Theorem~\ref{thm:rank_perfect2}(iii).
	\end{prop}	
\begin{proof} The group $G$ has the rank-independence property if and only if all proper subgroups of $G$ are cyclic. The conclusion follows from the description of minimal non-abelian groups by Miller and Moreno in \cite{mm}.
\end{proof}

\subsection*{Acknowledgements}
We are grateful to Peter Cameron for helpful discussions, and in particular for suggesting the problem of classifying the groups satisfying the rank-independence property; and to anonymous referees for helpful comments. The first author was supported by a St Leonard's International Doctoral Fees Scholarship and a School of Mathematics \& Statistics PhD Funding Scholarship at the University of St Andrews. The fourth author would like to thank the Isaac Newton Institute for Mathematical Sciences, Cambridge, for support and hospitality during the programme “Groups, Representations and Applications: New perspectives”, where work on this paper was undertaken. This work was supported by EPSRC grant no EP/R014604/1, and also partially supported by a grant from the Simons Foundation.


\begin{thebibliography}{99}

\bibitem{al}
C. Acciarri and A. Lucchini.
\newblock Graphs encoding the generating properties of a finite group.
\newblock {\em Math. Nachr.}, 293(9):1644--1674, 2020.

\bibitem{bereczky}
\'{A}. Bereczky.
\newblock Maximal overgroups of {S}inger elements in classical groups.
\newblock {\em J. Algebra}, 234(1):187--206, 2000.

\bibitem{magma}
W. Bosma, J. Cannon, and C. Playoust.
\newblock The {M}agma algebra system. {I}. {T}he user language.
\newblock {\em J. Symbolic Comput.}, 24(3-4):235--265, 1997.

\bibitem{bhrd}
J. N. Bray, D. F. Holt, and C. M. Roney-Dougal.
\newblock {\em The maximal subgroups of the low-dimensional finite classical
  groups}, volume 407 of {\em London Math.~Soc.~Lecture Note Ser.}
\newblock CUP, Cambridge, 2013.

\bibitem{GAPchar}
T. Breuer.
\newblock {\em {The GAP Character Table Library, Version 1.3.1}},
  2020.

\bibitem{bgk}
T. Breuer, R. M. Guralnick, and W. M. Kantor.
\newblock Probabilistic generation of finite simple groups. {II}.
\newblock {\em J. Algebra}, 320(2):443--494, 2008.

\bibitem{burnessgiudici}
T. C. Burness and M. Giudici.
\newblock {\em Classical groups, derangements and primes}, volume~25 of {\em
  Aust.~Math.~Soc.~Lecture Ser}.
\newblock CUP, Cambridge, 2016.

\bibitem{bgh}
T. C. Burness, R. M. Guralnick, and S. Harper.
\newblock The spread of a finite group.
\newblock {\em Ann. of Math. (2)}, 193(2):619--687, 2021.

\bibitem{burnessharper}
T. C. Burness and S. Harper.
\newblock On the uniform domination number of a finite simple group.
\newblock {\em Trans. Amer. Math. Soc.}, 372(1):545--583, 2019.


\bibitem{busa} S. Burris and H. P. Sankappanavar. {\em A course in universal algebra}. Grad. Texts in Math., vol. 78, Springer-Verlag, New York-Berlin, 1981.

\bibitem{buturlakin}
A. A. Buturlakin and M. A. Grechkoseeva.
\newblock The cyclic structure of maximal tori in finite classical groups.
\newblock {\em Algebra Logika}, 46(2):129--156, 2007.

\bibitem{cc} P. J. Cameron and P. Cara. Independent generating sets and geometries for symmetric groups.
{\em J. Algebra}, 258(2):641--650, 2002.
	
\bibitem {c0} P. J. Cameron and S. Ghosh.  The power graph of a finite group. {\em Discrete Math.}, 311(13):1220--1222, 2011.
	
\bibitem{c1} P. J. Cameron. The power graph of a finite group, II. {\em J. Group
Theory}, 13(6):779--783, 2010.

\bibitem{c3} P. J. Cameron, H. Guerra, and \v{S}. Jurina. The power graph
of a torsion-free group. {\em J. Algebraic Combinatorics}, 49(1):83--98, 2019.

\bibitem {c4} P. J. Cameron and S. H. Jafari, On the connectivity and
independence number of power graphs of groups. {\em Graphs Combin.}, 36(3):895--904, 2020.

\bibitem {c5} P. J. Cameron and B. Kuzma. Between the enhanced power
graph and the commuting graph. \textup{Preprint, 2020, arXiv:2012.03789}.

\bibitem{c6} P. J. Cameron, P. Manna, and R. Mehatari. On finite groups whose power graph is a cograph.  {\em J. Algebra}, 591:59--74, 2022.







	
	





\bibitem{carterweyl}
R. W. Carter.
\newblock Conjugacy classes in the {W}eyl group.
\newblock {\em Compositio Math.}, 25:1--59, 1972.

\bibitem{ATLAS}
J. H. Conway, R. T. Curtis, S. P. Norton, R. A. Parker, and R. A. Wilson.
\newblock {\em Atlas of finite groups}.
\newblock OUP, Eynsham, 1985.

\bibitem{defranceschi}
G. De~Franceschi.
\newblock {\em Centralizers and conjugacy classes in finite classical groups}.
\newblock PhD thesis, University of Auckland, 2018.


\bibitem{deriziotis81}
D. I. Deriziotis.
\newblock Centralizers of semisimple elements in a {C}hevalley group.
\newblock {\em Comm. Algebra}, 9(19):1997--2014, 1981.

\bibitem{galt}
A. A. Gal't.
\newblock Strongly real elements in finite simple orthogonal groups.
\newblock {\em Sibirsk. Mat. Zh.}, 51(2):241--248, 2010.

\bibitem{GAP}
The GAP~Group.
\newblock {\em {GAP -- Groups, Algorithms, and Programming, Version 4.11.1}},
  2021.

\bibitem{guangxiang}
Z. Guangxiang.
\newblock On self-normalizing cyclic subgroups.
\newblock In {\em The {A}rcata {C}onference on {R}epresentations of {F}inite
  {G}roups ({A}rcata, {C}alif., 1986)}, volume~47 of {\em Proc. Sympos. Pure
  Math.}, pp419--422. Amer. Math. Soc., Providence, RI, 1987.

\bibitem{guralnick}
R. Guralnick.
\newblock On the number of generators of a finite group.
\newblock {\em Arch. Math. (Basel)}, 53(6):521--523, 1989.

\bibitem{guralnickmalle}
R. Guralnick and G. Malle.
\newblock Products of conjugacy classes and fixed point spaces.
\newblock {\em J. Amer. Math. Soc.}, 25(1):77--121, 2012.

\bibitem{guma} R. Guralnick and G. Malle. Simple groups admit Beauville structures. {\em J. Lond. Math. Soc. (2)}, 85(3):694--721, 2012.

\bibitem{gpps}
R. Guralnick, T. Penttila, C. E. Praeger, and J. Saxl.
\newblock Linear groups with orders having certain large prime divisors.
\newblock {\em Proc. London Math. Soc. (3)}, 78(1):167--214, 1999.

\bibitem{guralnickkantor}
R. M. Guralnick and W. M. Kantor.
\newblock Probabilistic generation of finite simple groups.
\newblock {\em J. Algebra}, 234(2):743--792, 2000.

\bibitem{guralnicktracey}
R. M. Guralnick and G. Tracey.
\newblock On elements of finite groups contained in a unique maximal subgroup.
\newblock \textup{In preparation}.

\bibitem{sh} S. Harper.
\newblock Flexibility in generating sets of finite groups.
\newblock {\em Arch. Math. (Basel)}, 118(3):231--237, 2022.

\bibitem{hering}
C. Hering.
\newblock Transitive linear groups and linear groups which contain irreducible
  subgroups of prime order.
\newblock {\em Geometriae Dedicata}, 2:425--460, 1974.

\bibitem{hestenes}
M. D. Hestenes.
\newblock Singer groups.
\newblock {\em Canadian J. Math.}, 22:492--513, 1970.

\bibitem{huppert}
B. Huppert.
\newblock Singer-{Z}yklen in klassischen {G}ruppen.
\newblock {\em Math. Z.}, 117:141--150, 1970.

\bibitem{jordan}
C. Jordan.
\newblock Sur la limite de transitivit\'{e} des groupes non altern\'{e}s.
\newblock {\em Bull. Soc. Math. France}, 1:40--71, 1872/73.

\bibitem{kq} A. V. Kelarev and S. J. Quinn. A combinatorial property and power
  graphs of groups. 
 In  {\em Contributions to general algebra, 12 ({V}ienna, 1999)}, pages 229--235. Heyn, Klagenfurt, 2000.

\bibitem{kleidmanliebeck}
P. Kleidman and M. Liebeck.
\newblock {\em The subgroup structure of the finite classical groups}, volume
  129 of {\em London Math.~Soc.~Lecture Note Ser}.
\newblock CUP, Cambridge, 1990.


\bibitem{kleidmano8}
P. B. Kleidman.
\newblock The maximal subgroups of the finite {$8$}-dimensional orthogonal
  groups {$P\Omega^+_8(q)$} and of their automorphism groups.
\newblock {\em J. Algebra}, 110(1):173--242, 1987.

	



\bibitem{gene} A. Lucchini.  Generators and minimal normal subgroups. {\em Arch. Math. (Basel)}, 64(4):273--276, 1995. 

\bibitem{ind} A. Lucchini.  The independence graph of a finite group. {\em Monatsh. Math.}, 193(4):845--856, 2020.

\bibitem{quesgen} A. Lucchini.  Some questions on the number of generators of a finite group. {\em Rend. Sem. Mat. Univ. Padova}, 83:201--222, 1990.

\bibitem{cliq}  A. Lucchini and A.  Mar\'{o}ti. On the clique number of the generating graph of a finite group. {\em Proc. Amer. Math. Soc.}, 137(10):3207--3217, 2009.	



\bibitem{mallesaxlweigel}
G. Malle, J. Saxl, and T. Weigel.
\newblock Generation of classical groups.
\newblock {\em Geom. Dedicata}, 49(1):85--116, 1994.

\bibitem{malletesterman}
G. Malle and D. Testerman.
\newblock {\em Linear algebraic groups and finite groups of {L}ie type}, volume
  133 of {\em Cambridge Studies in Advanced Mathematics}.
\newblock CUP, Cambridge, 2011.

\bibitem{mm} G. A. Miller and H. C. Moreno.
Non-abelian groups in which every subgroup is abelian.
{\em Trans. Amer. Math. Soc.}, 4(4):398--404, 1903.

\bibitem{nicolas}
F. Nicolas.
\newblock A simple, polynomial-time algorithm for the matrix torsion problem.
\newblock \textup{Preprint, 2009, arXiv:0806.2068}.


	
	

\bibitem{e1} R. P. Panda, S. Dalal, and J. Kumar. On the enhanced power graph of a finite group. {\em Comm. Algebra}, 49(4):1697--1716, 2021.

\bibitem{praeger}
C. E. Praeger.
\newblock Primitive prime divisor elements in finite classical groups.
\newblock In {\em Groups {S}t. {A}ndrews 1997 in {B}ath, {II}}, volume 261 of
  {\em London Math. Soc. Lecture Note Ser.}, pp605--623. CUP, Cambridge, 1999.



\bibitem{wrscott}
W. R. Scott.
\newblock {\em Group theory}.
\newblock Dover Publications, Inc., New York, second edition, 1987.

\bibitem{seitzstructure}
G. M. Seitz.
\newblock On the subgroup structure of classical groups.
\newblock {\em Comm. Algebra}, 10(8):875--885, 1982.


\bibitem{tracey}
G. M. Tracey.
\newblock Minimal generation of transitive permutation groups.
\newblock {\em J. Algebra}, 509:40--100, 2018.

\bibitem{veldkamp}
F. D. Veldkamp.
\newblock Regular elements in anisotropic tori.
\newblock In {\em Contributions to algebra (collection of papers dedicated to
  {E}llis {K}olchin)}, pp389--424. Academic Press, New York, 1977.

\bibitem{wall}
G. E. Wall.
\newblock On the conjugacy classes in the unitary, symplectic and orthogonal
  groups.
\newblock {\em J. Austral. Math. Soc.}, 3:1--62, 1963.

\bibitem{weigel}
T. S. Weigel.
\newblock Generation of exceptional groups of {L}ie type.
\newblock {\em Geom. Dedicata}, 41(1):63--87, 1992.

\bibitem{wielandt}
H. Wielandt.
\newblock {\em Finite permutation groups}.
\newblock Academic Press, New York-London, 1964.

\bibitem{whi} J. Whiston. Maximal independent generating sets of the symmetric group. {\em J. Algebra}, 232(1):255--268, 2000.

\bibitem{sawhi} J. Whiston and J. Saxl. On the maximal size of independent generating sets of $PSL_2(q).$ {\em J.
Algebra},  258(2):651--657, 2002.

\bibitem{z1} S. Zahirovi\'{c}. The power graph of a torsion-free group determines
the directed power graph. {\em Discrete Appl. Math.}, 305:109--118, 2021.

\bibitem{z2}  S. Zahirovi\'{c}, I. Bo\v{s}njak, and R. Madar\'{s}z. A study of enhanced
power graphs of finite groups. {\em J. Algebra Appl.}, 19(4):2050062, 2020.

\bibitem{jpz} J. P. Zhang. 
On finite groups all of whose elements of the same order are conjugate in their automorphism groups.
{\em J. Algebra}, 153(1):22-–36, 1992.
\end{thebibliography}
\end{document}